\documentclass[10pt]{amsart}
\usepackage{amssymb}
\usepackage{amscd}
\usepackage[all]{xy}
\usepackage[colorlinks=true, linkcolor=red,
urlcolor=blue]{hyperref}

\newcounter{TmpEnumi}
\numberwithin{equation}{section}

\def\today{\number\day\space\ifcase\month\or
 January\or February\or
   March\or April\or May\or June\or
    July\or August\or September\or
   October\or November\or December\fi\
     \number\year}
\theoremstyle{definition}
\newtheorem{thm}{Theorem}[section]
\newtheorem{lem}[thm]{Lemma}
\newtheorem{prp}[thm]{Proposition}
\newtheorem{dfn}[thm]{Definition}
\newtheorem{cor}[thm]{Corollary}

\newtheorem{ctn}[thm]{Construction}
\newtheorem{rmk}[thm]{Remark}
\newtheorem{ntn}[thm]{Notation}
\newtheorem{exa}[thm]{Example}

\newtheorem{qst}[thm]{Question}

\newcommand{\beq}{\begin{equation}}
\newcommand{\eeq}{\end{equation}}
\newcommand{\beqr}{\begin{eqnarray*}}
\newcommand{\eeqr}{\end{eqnarray*}}
\newcommand{\bal}{\begin{align*}}
\newcommand{\eal}{\end{align*}}
\newcommand{\bei}{\begin{itemize}}
\newcommand{\eei}{\end{itemize}}
\newcommand{\limi}[1]{\lim_{{#1} \to \infty}}

\newcommand{\af}{\alpha}

\newcommand{\dt}{\delta}
\newcommand{\ep}{\varepsilon}

\newcommand{\et}{\eta}

\newcommand{\ld}{\lambda}

\newcommand{\Z}{{\mathbb{Z}}}

\newcommand{\N}{{\mathbb{Z}}_{> 0}}
\newcommand{\Nz}{{\mathbb{Z}}_{\geq 0}}

\newcommand{\Cu}{{\operatorname{Cu}}}

\newcommand{\id}{{\operatorname{id}}}
\newcommand{\ev}{{\operatorname{ev}}}

\newcommand{\spec}{{\operatorname{sp}}}

\newcommand{\supp}{{\operatorname{supp}}}
\newcommand{\rank}{{\operatorname{rank}}}

\newcommand{\card}{{\operatorname{card}}}
\newcommand{\Aut}{{\operatorname{Aut}}}
\newcommand{\Ad}{{\operatorname{Ad}}}

\newcommand{\QT}{{\operatorname{QT}}}
\newcommand{\T}{{\operatorname{T}}}
\newcommand{\W}{{\operatorname{W}}}
\newcommand{\tr}{{\operatorname{tr}}}

\newcommand{\rc}{{\operatorname{rc}}}

\newcommand{\U}{{\operatorname{U}}}

\newcommand{\cH}{{\mathcal{H}}}

\newcommand{\cK}{{\mathcal{K}}}

\newcommand{\dirlim}{\varinjlim}

\newcommand{\Mi}{M_{\infty}}

\newcommand{\andeqn}{\qquad {\mbox{and}} \qquad}

\newcommand{\Wolog}{Without loss of generality}

\newcommand{\tfae}{the following are equivalent}

\newcommand{\ca}{C*-algebra}

\newcommand{\hm}{homomorphism}

\newcommand{\hsa}{hereditary C*-subalgebra}

\newcommand{\pj}{projection}

\newcommand{\mvnt}{Murray-von Neumann equivalent}

\newcommand{\CGAa}{C^* (G, A, \af)}

\newcommand{\cfn}{continuous function}

\renewcommand{\S}{\subset}

\newcommand{\I}{\infty}

\newcommand{\Lem}[1]{Lemma~\ref{#1}}
\newcommand{\Def}[1]{Definition~\ref{#1}}
\newcommand{\Thm}[1]{Theorem~\ref{#1}}

\newcommand{\Ntn}[1]{Notation~\ref{#1}}

\RequirePackage[normalem]{ulem} 
\RequirePackage{color}\definecolor{RED}{rgb}{1,0,0}\definecolor{BLUE}{rgb}{0,0,1} 
\RequirePackage{listings} 
\RequirePackage{color} 
\lstdefinelanguage{DIFcode}{ 
  moredelim=[il][\color{red}\sout]{\%DIF\ <\ }, 
  moredelim=[il][\color{blue}\uwave]{\%DIF\ >\ } 
} 
\lstdefinestyle{DIFverbatimstyle}{ 
	language=DIFcode, 
	basicstyle=\ttfamily, 
	columns=fullflexible, 
	keepspaces=true 
} 
\lstnewenvironment{DIFverbatim}{\lstset{style=DIFverbatimstyle}}{} 
\lstnewenvironment{DIFverbatim*}{\lstset{style=DIFverbatimstyle,showspaces=true}}{} 

\title[Radius of comparison of the crossed product]{
The radius of comparison of the crossed product by a tracially strictly approximately inner  action.}  

\author{M. Ali Asadi-Vasfi}

\curraddr{School of Mathematics, Statistics and Computer Science,
College of Science, University of Tehran, Tehran, Iran.}
\email[]{Asadi.ali@ut.ac.ir}

\date{\today}

\subjclass[2010]{Primary 46L55;
 Secondary 19K14; 46L80.}

\begin{document}
\maketitle
\begin{center}
Dedicated to N. Christopher Phillips on the occasion of his 64th birthday.
\end{center}
\begin{abstract}
Let $G$ be a finite group,
let $A$ be an  infinite-dimensional stably finite simple unital C*-algebra,
and let $\alpha \colon G \to \Aut (A)$
be a tracially strictly approximately  inner action of $G$ on $A$.
Then the radius of comparison satisfies
$\rc (A) \leq   \rc \big( \CGAa \big)$ and if $C^*(G, A, \alpha)$ is simple, then
$\rc (A) \leq   \rc \big( \CGAa \big) \leq \rc (A^{\alpha})$.
Further, 
the inclusion of $A$ in~$C^*(G, A, \alpha)$
 induces an isomorphism from the purely positive part of the
 Cuntz semigroup $\Cu (A)$ to its image in $\Cu (\CGAa)$.
 If $\alpha$ is  strictly approximately inner,
 then in fact $\Cu (A) \to \Cu (\CGAa)$ is an ordered semigroup isomorphism onto
its range. Also, for every finite group~$G$ and for every  $\eta \in \left(0, \frac{1}{\card (G)}\right)$,
we construct a simple separable unital AH~algebra $A$ with stable rank one 
and a strictly approximately  inner action $\alpha \colon G \to \Aut (A)$  such that:
\begin{enumerate}
\item
$\alpha$ is pointwise outer and doesn't have the weak tracial Rokhlin property.
\item 
$\rc (A) =\rc \left(C^*(G, A, \alpha)\right)= \eta$.
\end{enumerate}
\end{abstract}

\tableofcontents

\section{Introduction}\label{Sec_Intro}
Comparison theory of projections plays an important role in the type classification of factors.
A C*-algebra might have few or no projections. 
In this case comparison theory of projections may say 
nothing about the structure of the C*-algebra. 
The appropriate substitute for projections is positive elements.
This idea was first introduced by Cuntz in \cite{Cun78} for the purpose  of studying dimension functions on simple C*-algebras.
Later, the radius of comparison of C*-algebras, based on the Cuntz semigroup,
was introduced by Andrew S.~Toms
in Section~6 of~\cite{Tom06} to study  exotic examples of simple amenable
C*-algebras that are not $\mathcal{Z}$-absorbing.
In the commutative setting, it
is well known that the radius of comparison of $C(X)$ is bounded above 
by one half the covering dimension of $X$ \cite{BRTTW12, EN13}. Also, the comparison theory can be viewed as a non-commutative dimension
theory \cite{AA20}.
Remarkable progress has been achieved on the comparison theory
in \cite{BRTTW12, HP19, Ni14, Ph16, Ph14}. 
In this paper, we consider the relation between comparison
theory in a simple C*-algebra $A$ and in the crossed product of $A$ by
a finite group, under a tracially strictly approximately inner action.

The  Cuntz semigroup is a key ingredient in the Elliott program for the classification of C*-algebras \cite{ET08, Tom08}.
We refer to~\cite{APT11, TT15} for many aspects of the Cuntz semigroup.
It is generally large and complicated.
Among simple nuclear \ca{s},
the classifiable ones are those whose Cuntz semigroups
are easily accessible (Section~5 of \cite{APT11}).
With the near completion of the Elliott program,
nonclassifiable \ca{s} attract more attention (see \cite{AGP19, HP19, Ph16})
and the Cuntz semigroup is the main additional available invariant.
The goal of this paper
is to conduct an investigation  beyond classifiable C*-algebras by
considering crossed products by finite groups of simple C*-algebras which may not have strict comparison of positive elements, 
although there are still many pieces of the puzzle that have to be put together.

Section~{\ref{Sec_Prelim}} is devoted to
 providing a quick overview of the background we need here. In Section~\ref{Sec_Approx_Innner} and Section~\ref{Sec_Tracial_Approx_Innner}, we give a standard definition for 
a strictly approximately inner action of a finite group on a unital C*-algebra
and its tracial analog, tracial strict approximate innerness.
 (See Definition~\ref{D_9719_StrictAppInn} and Definition~\ref{D_9731_TrStrAppInn}.)
To provide these definitions, we basically  give 
a generalization of the Definition~1.3 of \cite{Ph15}
and 
a  generalization of Definition~3.2 of~\cite{Ph11}
to nonabelian finite groups and not necessarily separable C*-algebras, and then omit some of their assumptions. 
(See Definition~\ref{D_9719_AppRep_Phi} and Definition~\ref{D_9731_NonAbTrAppRep}.)
Clearly, in the setting of abelian groups and  simple  separable unital C*-algebras,
approximate representability in the sense of Phillips (Definition~3.2 of~\cite{Ph11}),
approximate representability in the sense of Izumi (Definition~3.6(2) of \cite{Iz1}),
approximate representability in the sense of Osaka and Teruya (Definition~3.4 of \cite{LO19}),
and
approximate representability in our sense (Definition~\ref{D_9719_AppRep_Phi}) are the same.
There is no apparent reason to
believe that strict approximate innerness implies approximate
representability. 
No examples seem to be known.

Also, we prove that if $G$ is a finite group, $A$ is an infinite-dimensional simple  unital stably finite C*-algebra,
and $\alpha \colon  G \to \Aut(A)$ be a tracially strictly  approximately  inner action, then 
the radii of comparison of $A$ and 
the crossed product are related by
\[
\rc (A) \leq \rc \big( \CGAa\big).
\]
Getting the reverse inequality is likely to be very difficult.
When we further assume that $C^*(G, A, \alpha)$ is simple, then
the radii of comparison of $A$, the crossed product, and the fixed point algebra
are related by
\[
\rc (A) \leq \rc \big(C^*(G, A, \alpha)\big) \leq \rc (A^{\alpha}).
\]   
In fact,
we prove a much stronger result,
relating the Cuntz semigroups
(see Theorem~\ref{WC_injectivity} and Theorem~\ref{WC_plus_injectivity}).
We show that
the inclusion of $A$ in~$C^*(G, A, \alpha)$
induces an isomorphism from $\Cu (A)$  to its range in $\Cu (C^*(G, A, \alpha))$
if $\alpha$ is a strictly approximate inner action of a finite group $G$ on a unital C*-algebra $A$.
It induces an isomorphism from  $\Cu_+ (A) \cup \{0\}$ to its range 
if $\alpha$ is a tracially  strictly approximately inner action of a finite group $G$ 
on an infinite-dimensional simple  unital stably finite C*-algebra $A$.
   
Suppose further that $G$ is abelian, $\alpha$ is tracially approximately representable, and $C^*(G, A, \alpha)$ is simple. 
Let $\widehat{\alpha}$ be the dual action of $\widehat{G}$ on $C^*(G, A, \alpha)$. 
Then $\widehat{\alpha}$  has the tracial Rokhlin property by Theorem~3.11 of \cite{Ph11}. Therefore our results are immediate from 
Takai duality \cite{Tak75} and the main results in \cite{AGP19}.
But the hypothesis in our results is only that the action is tracially
strictly approximately inner.  
Also, the group in our work need not to be abelian. 
Thus, our results generalize the tracial Rokhlin property case
of \cite{AGP19}.
Note that the right version of the Rokhlin property for
actions of finite dimensional quantum groups was given in \cite{GKL19}
under the name ``spatial Rokhlin property''
and the tracial version of spatial Rokhlin property has not been defined yet. 
Also, even if we had the result of \cite{AGP19} for duals of finite dimensional quantum groups,
Theorem~4.12 of \cite{BSV17} would only give the current result assuming approximate representability.
 It is because of the fact that our actual hypotheses are weaker.
 Namely, strict approximate innerness is weaker than approximate representability and 
 our results cover also the tracial version of approximate representability.

We also prove in Proposition~\ref{Pr.In.Ro} that actions of finite groups on
stably finite unital C*-algebras $A$ with $0<\rc(A)<\infty$
 cannot simultaneously have the Rokhlin property and be strictly approximately inner.
We further prove in Proposition~\ref{Pr.Tracial.In.Ro} that  actions of finite groups on many  nonclassifiable simple  unital C*-algebras cannot simultaneously
 have the weak tracial Rokhlin property  and be tracially strictly approximately inner.  
Proposition~\ref{Pr.In.Ro} and Proposition~\ref{Pr.Tracial.In.Ro} fail when the C*-algebras have strict comparison of positive elements.
(See Example~\ref{Count.Examp}.)

Finally, in Section~\ref{Sec_Example}, for every finite group~$G$ and for every  $\eta \in \left(0, \frac{1}{\card (G)}\right)$, 
we construct a simple separable unital AH~algebra $A$ with stable rank one 
and an action $\alpha \colon G \to \Aut (A)$  such that:
\begin{enumerate}
\item
$\alpha$ is pointwise outer, strictly approximately  inner, and doesn't have the weak tracial Rokhlin property.
\item 
$\rc (A) =\rc \left(C^*(G, A, \alpha)\right)= \eta$.
\end{enumerate}
\subsection*{Acknowledgments}
Some parts of
this work were carried out during the research visit of the author to the University of Oregon 
from March 2018 to September 2019.
He is thankful to that institution for its hospitality with special thanks to Ben Elias and Sherilyn Schwartz.
The author would like to thank N. Christopher Phillips
for a number of  productive discussions and feedback. 
He is also grateful to Lawrence G. Brown for answering a question about an appropriate choice of notation,
Eusebio Gardella for pointing out the reference \cite{G20},
Nasser Golestani for pointing out Lemma~2.3 of \cite{HO13},
 Ilan Hirshberg for pointing out Example~5.11 of \cite{Iz1},
 Mahdi Moradi for his comment on the proof of Lemma~\ref{Approx_Commutnat}, 
Hiroyuki  Osaka for pointing out Proposition~3.17 of \cite{LO19}, Gábor Szabó for pointing out the reference \cite{BSV17},
and Hannes Thiel for point out Proposition~2.8 of \cite{Th20}.
\section{Preliminaries }\label{Sec_Prelim}
In this preliminary  section,
we collect
some information on the Cuntz semigroup, quasitraces,
and the radius of comparison  for easy reference and the convenience of the reader.
\begin{ntn}
Throughout, 
if $A$ is a \ca, or if $A = M_{\infty} (B)$
for a C*-algebra~$B$, we write $A_{+}$ for the
set of positive elements of $A$ and write $\U (A)$ for its
unitary group.
For $a \in A_+$, we denote by $(a -\ep)_+$ the function $\max(0,t-\ep)$ on the spectrum of $a$.
 Also, we denote by $\cK$ the algebra of compact operators on a separable
and infinite-dimensional Hilbert space $\cH$.
\end{ntn}
Part (\ref{Cuntz_def_property_a})
of the following definition is originally from~\cite{Cun78}
and Part (\ref{Cuntz_def_property_b}) is from~\cite{APT11}.
\begin{dfn}\label{Cuntz_def_property}
Let $A$ be a \ca.
\begin{enumerate}
\item\label{Cuntz_def_property_a}
For any $a, b \in M_{\infty} (A)_{+}$,
we say that $a$ is {\emph{Cuntz subequivalent to~$b$ in~$A$}},
written $a \precsim_{A} b$,
if there is a sequence $(c_n)_{n = 1}^{\infty}$ in $M_{\infty} (A)$
such that
\[
\limi{n} c_n b c_n^* = a.
\] 
If $a \precsim_{A} b$ and $b \precsim_{A} a$, 
we say that $a$ and $b$ are {\emph{Cuntz equivalent in~$A$}} and write $a \sim_{A} b$.
This relation is an equivalence relation,
and we write $\langle a \rangle_A$ for the equivalence class of~$a$.
We define $\W (A) = M_{\infty} (A)_{+} / \sim_A$,
together with the commutative semigroup operation
$\langle a \rangle_A + \langle b \rangle_A
 = \langle a \oplus b \rangle_A$
and the partial order
$\langle a \rangle_A \leq \langle b \rangle_A$
if $a \precsim_{A} b$.
We write $0$ for~$\langle 0 \rangle_A$.
Further, 
we take $\Cu (A) = \W (\cK \otimes A)$.
We write the classes as $\langle a \rangle_A$
for $a \in (\cK \otimes A)_{+}$.
\item\label{Cuntz_def_property_b}
Let $A$ and $B$ be C*-algebras
and let $\psi \colon A \to B$ be a \hm.
We use the same letter for the induced maps
$M_n (A) \to M_n (B)$
for $n \in \N$, 
$\Mi (A) \to \Mi (B)$, and 
$\cK \otimes A\to \cK \otimes B$.
We define
$\W (\psi) \colon \W (A) \to \W (B)$
and $\Cu (\psi) \colon \Cu (A) \to \Cu (B)$
by $\langle a \rangle_A \mapsto \langle \psi (a) \rangle_B$
for $a \in M_{\infty} (A)_{+}$
or $a \in (\cK \otimes A)_{+}$ as appropriate.
\end{enumerate}
\end{dfn}
The original version $\W(A)$ of the Cuntz semigroup  has some flaws
 which are fixed by considering $\Cu (A)$, although both invariants are closely related. 
 Nevertheless, the radius of comparison  is easier to deal with in terms
of $\W(A)$. 
Also, the usual notation for Cuntz subequivalence is $a \precsim b$.
Since we need to use Cuntz subequivalence with respect to different C*-algebras,
we include $A$ in the notation.

The following lemma is taken from~\cite{KR00, Ph14, Ror92}.
\begin{lem}\label{PhiB.Lem_18_4}
Let $A$ be a \ca.
\begin{enumerate}
\item\label{PhiB.Lem_18_4_10}
Let $a, b \in A_{+}$ and let $\eta > 0$.
If $\| a - b \| < \eta$, then:
\begin{enumerate}
\item\label{PhiB.Lem_18_4_10.a}
$(a - \eta)_{+} \precsim_A b$.
\item\label{Item_9420_LgSb_1_6}
For any $\ld > 0$,
we have $(a - \ld - \eta)_{+} \precsim_A (b - \ld)_{+}$.
\end{enumerate}
\item\label{PhiB.Lem_18_4_8}
Let $a \in A_{+}$ and let $\eta_1, \eta_2 > 0$.
Then
\[
\big( ( a - \eta_1)_{+} - \eta_2 \big)_{+}
 = \big( a - ( \eta_1 + \eta_2 ) \big)_{+}.
\]
\item\label{PhiB.Lem_18_4_11}
Let $a, b \in A_{+}$.
Then \tfae:
\begin{enumerate}
\item\label{PhiB.Lem_18_4_11.a}
$a \precsim_A b$.
\item\label{PhiB.Lem_18_4_11.b}
$(a - \eta)_{+} \precsim_A b$ for all $\eta > 0$.
\item\label{PhiB.Lem_18_4_11.c}
For every $\eta > 0$, there is $\dt > 0$ such that
$(a - \eta)_{+} \precsim_A (b - \dt)_{+}$.
\end{enumerate}
\end{enumerate}
\end{lem}
The following definition is originally from \cite{BH82}.
It is also Defnition~3.1 of \cite{Hag14}.
\begin{dfn}\label{quasitrace}
A {\emph{quasitrace}} on a unital C*-algebra $A$ is
a function $\rho \colon A \to \mathbb{C}$
such that  the following hold:
\begin{enumerate}
\item\label{quasitrace.a}
$\rho (a^* a) = \rho (a a^*) \geq 0$ for all $a \in A$.
\item\label{quasitrace.b}
$\rho (b + i c ) = \rho (b) + i \rho ( c ) $
for $b, c \in A_{\mathrm{sa}}$.
\item\label{quasitrace.c}
$\rho |_B$ is linear
for every commutative C*-subalgebra $B \subseteq A$.
\item\label{quasitrace.d}
There is a function $\rho_2 \colon M_2 (A) \to \mathbb{C}$
satisfying
(\ref{quasitrace.a}), (\ref{quasitrace.b}), and (\ref{quasitrace.c})
with $M_2 (A)$ in place of $A$,
and such that,
with $(e_{j, k})_{j, k = 1}^{2}$
denoting the standard system of matrix units in $M_2 (\mathbb{C})$,
for all $a \in A$ we have
\[
\rho (a) = \rho_2 (a \otimes e_{1, 1}).
\]
\end{enumerate}
A quasitrace $\rho$ on a unital \ca{}
is {\emph{normalized}} if $\rho (1) = 1$.
The set of normalized quasitraces on a unital C*-algebra $A$ is denoted by $\QT (A)$.
\end{dfn}

It was shown in the discussion after Proposition II.4.6 of~\cite{BH82} that 
$\QT ( A )\neq \varnothing$ for every stably finite unital C*-algebra $A$.  
We refer to \cite{BH82, Hag14} for more details about quasitraces. 

The following is Definition~6.1 of~\cite{Tom06},
except that we allow $r = 0$ in~(\ref{rc_dfn.a}).
This change makes no difference.

\begin{dfn}\label{rc_dfn}
Let $A$ be a stably finite unital C*-algebra.

\begin{enumerate}
\item\label{rc_dfn_a_0}
For $\rho \in \QT(A)$, define $d_{\rho} \colon \Mi (A)_{+} \to [0, \infty)$
by
\[
d_{\rho} (a) = \lim_{n \to \infty} \rho (a^{1/n}).
\]
\item\label{rc_dfn.a}
Let $r \in [0, \I)$.
We say that $A$ has {\emph{$r$-comparison}} if whenever
$a, b \in M_{\infty} (A)_{+}$ satisfy
\[
d_{\rho} (a) + r < d_{\rho} (b)\]
for all $\rho \in \QT (A)$,
then $a \precsim_A b$.
\item\label{rc_dfn.b}
The {\emph{radius of comparison}} of~$A$,
denoted ${\operatorname{rc}} (A)$, is
\[
\rc (A)
 = \inf \big( \big\{ r \in [0, \I) \colon
    {\mbox{$A$ has $r$-comparison}} \big\} \big)
\]
if it exists, and $\infty$ otherwise.
\end{enumerate}
\end{dfn}
Note that the limit in Definition~\ref{rc_dfn}(\ref{rc_dfn_a_0}) obviously exists if $|| a || \leq 1$. If $|| a || > 1$,
    the limit is the same as if one uses $|| a ||^{-1} a$ in place
    of $a$.
    Also, It was shown in Proposition~6.3 of~\cite{Tom06} that if $A$ is simple, then $A$ has ${\operatorname{rc}} (A)$-comparison.

The following proposition is taken from Proposition~6.2 of \cite{Tom06}.
\begin{prp}
\label{Prp6.2.Tom06}
Let $A$ and $B$ be  stably finite unital C*-algebras. Then:
\begin{enumerate}
\item
\label{Prp6.2.Tom06.a}
$\rc (A \oplus B) = \max \big(\rc (A), \rc(B)\big)$.
\item
\label{Prp6.2.Tom06.b}
$\rc (M_n \otimes A)= \frac{1}{n} \cdot \rc (A)$ for $n \in \N$.
\end{enumerate}

\end{prp}

The following proposition, which is a
special case of results in~\cite{BRTTW12} and is Theorem~12.4.4 of \cite{GKPT18},
 is an algebraic reformulation of the radius of comparison for simple C*-algebras.
\begin{prp}\label{rc_alg_def}
Let $A$ be a stably finite simple unital C*-algebra.
Then  $\rc(A)$ is the least number $s \in [0, \infty]$ such that
whenever $m,n \in \N$ satisfy $m/n > s$, and $a,b \in M_{\infty} (A)_+$ satisfy
\[
(n + 1)\langle a \rangle_ A + m\langle 1 \rangle_ A \leq \langle b \rangle_ A
\]
in $W(A)$, then $a \precsim_A b$.
\end{prp}
The following theorem, providing bounds on the radius of comparison of a full corner
in a \ca, is a special case of Theorem~2.18 in \cite{AGP19}. 
\begin{thm}\label{Ourcornertheorem}
Let $A$ be a stably finite unital \ca{}
and let $q$ be a full
projection in $A$.
Define
\[
\lambda
 = \inf \bigl( \bigl\{ \rho (q) \colon
     \rho \in \QT (A) \bigr\} \bigr)
\andeqn
\eta
 = \sup \bigl( \bigl\{ \rho (q) \colon
     \rho \in \QT (A) \bigr\} \bigr).
\]
Then $0 < \ld \leq \et \leq 1$ and
$
\frac{1}{\eta} \cdot \rc (A)
 \leq \rc ( q A q )
 \leq \frac{1}{\lambda} \cdot \rc (A).
$
\end{thm}
\section{The radius of comparison and strict approximate innerness}
\label{Sec_Approx_Innner}
In this section, we first prove that approximate representability in the sense of Izumi 
and approximate representability in the sense of Phillips are  the same in the setting of nonabelian finite groups. 
We then give the definition of strictly approximately inner actions 
of finite groups on  unital C*-algebras. 
Then we find a unital linear map  from $C^*(G, A, \alpha)$
to $A$ which is an approximate homomorphism on finite sets and whose restriction to $A$ is the identity map.
Also, we show that 
the map $\W (\iota) \colon \W (A) \to \W (\CGAa)$
and 
the map $\Cu (\iota) \colon \Cu (A) \to \Cu (\CGAa)$
which are induced by the inclusion map of $A$ in $C^*(G, A, \alpha)$ are injective.
We further find a lower bound and an upper bound for the radius of comparison of $C^*(G, A, \alpha)$.
\begin{dfn}
Let $A$ be a C*-algebra. We  denote the set of all bounded
sequences in A with the supremum norm and pointwise operations by $l^{\infty} (\N, A)$.
Define
\[
c_0 (\N, A)=\left\{ (a_n)_{n \in \N} \in  l^{\infty} (\N, A) \colon \lim_{n \to \infty} \| a_n \|=0  \right\}.
\]
Clearly $c_0 (\N, A)$ is an ideal of $l^{\infty} (\N, A)$. Then define
\[
A_{\infty} = l^{\infty} (\N, A)/ c_0 (\N, A).
\]
We identify $A$ with the C*-subalgebra of $A_{\infty}$ consisting of the equivalence classes of
constant sequences.
Let $\alpha \colon G \to \Aut (A)$ be an action of a finite group $G$ on a C*-algebra $A$.
 We write $\alpha_{\infty} \colon G \to \Aut (A_{\infty})$ for the induced action.
\end{dfn}
The following definition 
is a generalization of the Definition~3.6(1) of \cite{Iz1} 
to nonabelian finite groups  in terms of the  sequence algebra.
\begin{dfn}
\label{Df_Strog_App_Inn_IZ}
Let $\alpha \colon G \to \Aut (A)$ be an action of a finite group $G$ on a unital C*-algebra $A$.
The action $\alpha$ is said to be \emph{strongly approximately inner} if 
there exist unitaries $w(g) \in A_{\infty}$ such that:
\begin{enumerate}
\item
$\alpha_{{\infty}, g} (x) = w(g)x w(g)^*$ for all  $x \in A$ and all $g \in G$.
\item
$\alpha_{{\infty}, g} (w(h)) = w(ghg^{-1})$ for all $g, h \in G$.
\end{enumerate}
\end{dfn}

The following definition is a generalization of the Definition~3.6(2) of \cite{Iz1} 
to nonabelian finite groups  in terms of the  sequence algebra. 
It was also pointed out by Izumi in Remark~3.7 of \cite{Iz1}.
\begin{dfn} \label{Izum_Def_app_Rep}
Let $\alpha \colon G \to \Aut (A)$ be an action of a finite group $G$ on a unital C*-algebra $A$.
The action $\alpha$ is said to be \emph{approximately representable} if 
there exists a unitary representation $w$ of $G$ in $A_{\infty}$ such that:
\begin{enumerate}
\item
$\alpha_{{\infty}, g} (x) = w(g)x w(g)^*$ for all  $x \in A$ and all $g \in G$.
\item
$\alpha_{{\infty}, g} (w(h)) = w(ghg^{-1})$ for all $g, h \in G$.
\end{enumerate}
\end{dfn}
The following is a generalization of the Definition 1.3 of \cite{Ph15} to nonabelian finite groups.
\begin{dfn}
\label{D_9719_AppRep_Phi}
Let $A$ be a unital \ca,
let $G$ be a finite group,
and let $\af \colon G \to \Aut (A)$
be an action of $G$ on~$A$.
We say that $\af$ is
{\emph{approximately representable}}
if for every finite set  $F \S A$ and every $\ep > 0$,
there are  $z_g \in \U (A)$ for $g \in G$ such that:
\begin{enumerate}
\item \label{D_9719_AppRep.1}
$\| \af_g (a) - z_g a z_g^* \| < \ep$ for all $g \in G$ and all $a \in F$.
\item \label{D_9719_AppRep.2}
$\| z_g z_h - z_{gh}\| < \ep$ for all $g, h \in G$.
\item \label{D_9719_AppRep.3}
 $\| \af_g (z_h) - z_{g h g^{-1}}\| < \ep$ for all $g, h \in G$.
\end{enumerate}
\end{dfn}
Obviously, in the abelian case,
 Condition (\ref{D_9719_AppRep.3})
reduces to $\| \af_g (z_h) - z_{h}\| < \ep$.
The following lemma  is only intended to show that Lemma~3.1 of \cite{Ph11} is still true in the setting of nonabelian groups.
\begin{lem}\label{Lem.PZ.20200623}
Let $A$ be a separable unital C*-algebra and let $G$ be a finite group (not necessarily abelian).
Let $\alpha \colon G \to \Aut (A)$ be an action.
Then $\alpha$ is approximately representable in the sense of Izumi (Definition~\ref{Izum_Def_app_Rep}) if and only if
$\alpha$ is approximately representable in the sense of Phillips (Definition~\ref{D_9719_AppRep_Phi}).
\end{lem}
\begin{proof}
We first prove the backwards implication. Since $A$ is separable, 
we can choose finite sets $F_n \subset A$ such that
$F_1 \subseteq F_2 \subseteq \ldots$ and $\overline{\bigcup_{n=1}^{\infty} F_n} = A$.
Applying  Definition~\ref{D_9719_AppRep_Phi} for each $F_n$, we get 
 $z_{n}(g) \in \U (A)$ for $g \in G$ such that:
\begin{enumerate}
\item \label{D_9719_AppRep.1'}
$\| \af_g (a) - z_{n}(g) a z_{n}(g)^* \| < \frac{1}{n}$ for all $a \in F_n$ and all $g \in G$.
\item \label{D_9719_AppRep.2'}
$\| z_{n}(g) z_{n}(h) - z_{n}(gh)\| < \frac{1}{n}$ for all $g, h \in G$.
\item \label{D_9719_AppRep.3'}
 $\| \af_g (z_{n}(h)) - z_{n}(g h g^{-1})\| < \frac{1}{n}$ for all $g, h \in G$.
\end{enumerate}
Now, for $g \in G$, set $w^{(0)}_{ g}= (z_{n}(g))_{n\in \N}$ as an element in $l^{\infty} (\N, A)$
and let $w(g)$ be its image in $A_{\infty}$.
Then, for all $a \in \cup_{m=1}^{\infty} F_m$, we have
\begin{equation} \label{Eq1.20200510}
\lim_{n \to \infty}
 \left\| \alpha_{g} (a) - z_{n}(g) a z_{n}(g)^*\right\|=0.
\end{equation}
Let $\ep>0$ and let $x \in A$. Choose $m \in \N$ and $a \in F_m$ such that 
\begin{equation}  \label{Eq2.20200510}
\|x - a \|<\tfrac{\ep}{3}.
\end{equation}
By (\ref{Eq1.20200510}), we can choose $N \in \N$ such that $N \geq m$ and for all $n \geq N$
\begin{equation} \label{Eq3.20200510}
\left\| \alpha_{g} (a) - z_{n}(g) a z_{n}(g)^*\right\| < \frac{\ep}{3}.
\end{equation}
For all $n \geq N$, we use (\ref{Eq2.20200510}) and (\ref{Eq3.20200510}) to get
\begin{align*}
\left\| \alpha_{g} (x) - z_{n}(g) x z_{n}(g)^*\right\|
&<
\left\| \alpha_{g} (x) - \alpha_g (a)\right\|
+
\left\| \alpha_{g} (a) - z_{n}(g) a z_{n}(g)^*\right\|
\\
& \qquad+
\|z_{n}(g)\| \cdot \|  a  -  x \| \cdot \|z_{n}(g)^*\|
\\
&<
\frac{\ep}{3} + \frac{\ep}{3} + \frac{\ep}{3} =\ep.
\end{align*}
Thus,
\[
\lim_{n \to \infty}
 \left\| \alpha_{g} (x) - z_{n}(g) x z_{n}(g)^*\right\|=0.
 \]
Therefore, for all $x \in A$, and all $g \in G$,
\[
\alpha_{\infty, g} (x) = w(g) x w(g)^*.
\]
By (\ref{D_9719_AppRep.2'}) and (\ref{D_9719_AppRep.3'}), we further have, for all $g, h \in G$,
\[
\lim_{n \to \infty}
 \left\| z_{n}(g) z_{n}(h) - z_{n}(gh)\right\|=0
\quad 
\mbox{ and }
\quad 
\lim_{n \to \infty}
 \left\| \alpha_{g} (z_n (h)) - z_{n}(g h g^{-1})\right\|=0.
 \]
Thus, the map $g \mapsto w(g)$ is a unitary representation from $G$ to $A_{\infty}$ and 
$\alpha_{\infty, g} (w(h)) = w(g h g^{-1})$ for all $g, h \in G$. 

To prove the forwards implication, consider $w(g)$ in Definition~\ref{Izum_Def_app_Rep} 
as being presented by a sequence $(w_n (g))_{n \in \N}$.
Then 
\[
\lim_{n \to \infty} \|w_n(g) w_n (g)^* -1 \|=0
\qquad
\mbox{ and }
\qquad
\lim_{n \to \infty} \|w_n (g)^* w_n(g) -1 \|=0.
\]
Now using semiprojectivity of $C(S^1)$, we can find a sequence $(u_n (g))_{n \in \N}$ of unitaries in $A$ such that
$\lim_{n \to \infty} \|u_n(g) - w_n(g)\|=0$.
Then we still have, for all $a \in A$ and all $g, h\in G$,
\[
\lim_{n \to \infty} \| \alpha_{g} (a) - u_n (g) a u_n (g)\|=0
\qquad
\mbox{ and }
\qquad
\lim_{n \to \infty} \| \alpha_{g} (u_n (h)) - u_n (g h g^{-1})\|=0.
\]
Now take $z(g)= u_n (g)$ for some fixed and sufficiently large $n$.
Therefore Condition~(\ref{D_9719_AppRep.1'}) through Condition~(\ref{D_9719_AppRep.3'})
 in Definition~\ref{D_9719_AppRep_Phi} hold.
\end{proof}
Now, we  omit Condition (\ref{D_9719_AppRep.3}) in \Def{D_9719_AppRep_Phi} and make the following definition.
\begin{dfn}\label{D_9719_StrictAppInn}
Let $A$ be a unital \ca,
let $G$ be a finite group,
and let $\af \colon G \to \Aut (A)$
be an action of $G$ on~$A$.
We say that $\af$ is
{\emph{strictly approximately inner}}
if for every finite set $F \subset A$ and every $\ep > 0$,
there are  $z_g \in \U (A)$ for $g \in G$ such that:
\begin{enumerate}
\item \label{D_9719_StrictAppInn.1}
$\| \af_g (a) - z_g a z_g^* \| < \ep$ for all $g \in G$ and all $a \in F$.
\item \label{D_9719_StrictAppInn.2}
$\| z_g z_h - z_{gh}\| < \ep$ for all $g, h \in G$.
\end{enumerate}
\end{dfn}
\begin{rmk}
Obviously, 
approximate representability implies strict approximate innerness.
But there is no reason to expect that the converse is true and the naive method to prove the converse fails.
It seems that counterexamples are hard to find.
However, there is evidence of existing such examples. Namely, assume the notation in the proof of Lemma~\ref{Lem.PZ.20200623}.
Since $w_g^{(0)}$ is not a constant sequence, we cannot expect naively to get $\alpha_{\infty, g} (w(h)) = w(g h g^{-1})$ from 
(\ref{D_9719_AppRep.1'}) and (\ref{D_9719_AppRep.2'}).
\end{rmk}
We provide a stronger version of Definition~\ref{D_9719_StrictAppInn} in the following lemma.
\begin{lem}\label{Lem.s.a.i_hom}
Let $A$ be a \ca,
let $G$ be a finite group,
and let $\af \colon G \to \Aut (A)$
be an action of $G$ on~$A$. Then $\af$ is strictly approximately inner
if and only if
for every finite set $F \S A$ and every $\ep > 0$,
there is a \hm{} $g \mapsto z_g$
from $G$ to $\U (A)$ such that, for all $g \in G$ and all $a \in F$,
\[
\| \af_g (a) - z_g a z_g^* \| < \ep.
\] 
\end{lem}

\begin{proof}
It follows from semiprojectivity of $C^*(G)$. 
\end{proof}
The following lemma provides a version of  Definition~\ref{D_9719_StrictAppInn} which is stated in terms of the  sequence algebra. 
\begin{lem}
Let $A$ be a separable unital C*-algebra and let $G$ be a finite group.
Let $\alpha \colon G \to \Aut (A)$ be an action.
Then $\alpha$ is strictly approximately inner if and only if
there exists a unitary representation $w$ of $G$ in $A_{\infty}$ such that
\[
\alpha_{{\infty} g} (x) = w(g)x w(g)^*
\]
 for all  $x \in A$ and all $g \in G$.
\end{lem}
\begin{proof}
The proof is essentially the same as that of Lemma~\ref{Lem.PZ.20200623}.
\end{proof}
 The following definition is a characterization of strong
approximate innerness in terms of elements in $A$ instead of $A_{\infty}$.
\begin{dfn}
\label{Df_storng_App_inner}
Let $A$ be a unital \ca,
let $G$ be a finite group,
and let $\af \colon G \to \Aut (A)$
be an action of $G$ on~$A$.
We say that $\af$ is
{\emph{strongly approximately inner}}
if for every finite set $F \subset A$ and every $\ep > 0$,
there are  $z_g \in \U (A)$ for $g \in G$ such that:
\begin{enumerate}
\item \label{Df_storng_App_inner_1}
$\| \af_g (a) - z_g a z_g^* \| < \ep$ for all $g \in G$ and all $a \in F$.
\item \label{Df_storng_App_inner_2}
$\| \af_g (z_h) - z_{g h g^{-1}}\| < \ep$ for all $g, h \in G$.
\end{enumerate}
\end{dfn}
The following lemma is a generalization of Lemma~3.2 of \cite{G20} to arbitrary finite groups.
\begin{lem}
Let $A$ be a separable unital C*-algebra and let $G$ be a finite group.
Let $\alpha \colon G \to \Aut (A)$ be an action.
Then $\alpha$ is strongly approximately inner in the sense of Definition~\ref{Df_storng_App_inner} if and only if
$\alpha$ is strongly approximately inner in the sense of Definition~\ref{Df_Strog_App_Inn_IZ}.
\end{lem}
\begin{proof}
The proof is essentially the same as that of Lemma~\ref{Lem.PZ.20200623}.
\end{proof}
Obviously, approximate representability implies strong approximate innerness
and  the converse is not true (see the following example). 
Also, we do not know whether strong approximate innerness  together with strict approximate innerness implies 
approximate representability.
Further, one must be careful about the technical differences between strong approximate innerness (Definition~\ref{Df_storng_App_inner})  
and strict approximate innerness (Definition~\ref{D_9719_StrictAppInn}).
To distinguish them from each other, we provide the following example.
\begin{exa}
\label{Eample_Iz_me}
Let $\alpha \colon \mathbb{Z}/2 \mathbb{Z} \to \Aut(\mathcal{O}_3)$ be the action as in Example~5.11 of \cite{Iz1}.
Then:
\begin{enumerate}
\item
$\alpha$ is strongly approximately inner  by Proposition~5.6(2) of \cite{Iz1}.
\item
 The dual action of $\alpha$ does not have the Rohlin property by Theorem~3.13 of \cite{Iz1}.
 Therefore, $\alpha$ is not approximately representable.
\item
$\alpha$ is not strictly approximately inner. 
Namely, if we assume that $\alpha$ is strictly approximately inner, then, by 
Theorem~\ref{WC_injectivity}(\ref{WC_injectivity_a}), the map 
\[
K_0 (\iota) \colon K_{0} (\mathcal{O}_3) \to K_0 \big(C^*(\mathbb{Z}/2 \mathbb{Z}, \mathcal{O}_3, \alpha)\big),
\]
induced by the inclusion map, is injective. On the other hand, we have 
\[
K_{0} (\mathcal{O}_3) \cong \mathbb{Z}/2 \mathbb{Z}
\quad
\mbox{ and }
\quad
K_0 \big(C^*(\mathbb{Z}/2 \mathbb{Z}, \mathcal{O}_3, \alpha)\big) \cong \mathbb{Z}.
\]
This is a contradiction.
\end{enumerate}
\end{exa}
\begin{lem}\label{Lem_inlim_in_ap}
Let $G$ be a finite group, 
let $\big((G,A_j, \alpha^{(j)})_{j \in J} , (\varphi_{j,k})_{j\leq k} \big)$ be a direct system of $G$-algebras, 
let $A$ be the direct limit of the $A_j$,
and let $\alpha \colon G \to \Aut(A)$ be the direct limit of the $\alpha^{(j)}$. 
If $\alpha^{(j)}$ is inner for each $j$, then $\alpha$ is strictly approximately inner.
\end{lem}

\begin{proof}
Let $\ep > 0$, let $m \in \N$, and let $a_1, a_2, \ldots, a_m \in A$.
Choose $n \in \N$
and $b_1, b_2, \ldots, b_m \in A_n$
such that, for $j = 1, 2, \ldots, m$,
\begin{equation}\label{Eq1.24.07.19}
\| \varphi_{\infty, n} (b_j) - a_j \| < \frac{\ep}{2}.
\end{equation} 
Since $\alpha^{(n)}$ is inner,
there is a homomorphism $g \mapsto z_g$ from $G$ to  $\U (A_n)$ 
such that, for all $j\in \{ 1,2,\cdots, n\}$ and all $g \in G$,
\begin{equation}\label{Eq2.24.07.19}
\alpha^{(n)}_g (b_j) = z_g b_j z^*_g.
\end{equation}
Define $w_g = \varphi_{\I, n} (z_g)$. Therefore, 
using $\varphi_{\infty, n} \circ \alpha^{(n)}_g = \alpha_g \circ \varphi_{\infty, n}$ 
and (\ref{Eq2.24.07.19}) at the first  step and using (\ref{Eq1.24.07.19}) at the second step,
\begin{align*}
\| \alpha_g (a_j) - w_g a_j w^*_g\|
&\leq
\| \alpha_g (a_j) - \alpha_g (\varphi_{\infty, n} (b_j))\|
\\& \qquad +
\| \varphi_{\infty, n} (z_g) \varphi_{\infty, n} (b_j) \varphi_{\infty, n} (z^*_g)
-
w_g a_j w^*_g
 \|
\\\notag
&< \frac{\ep}{2} + \frac{\ep}{2} = \ep.
\end{align*}
This completes the proof.
\end{proof}
\begin{ntn}\label{N_9408_StdNotation_CP}
Let $A$ be a unital C*-algebra and let $\af \colon G \to \Aut (A)$ be an action
of a finite group $G$ on $A$. 
For $g \in G$,
we let $u_g$ be the element of $C_{\mathrm{c}} (G, A, \af)$
which takes the value $1$ at $g$
and $0$ at the other elements of~$G$.
We use the same notation for its image in $C^* (G, A, \af)$.
Also, for each $g \in G$, we define the map 
$ E_g \colon  \CGAa \to A$ by 
\[
E_g (a) = a_g,
\]
where $a= \sum_{g \in G} a_g u_g$.
We also denote by $A^{\alpha}$ the fixed
point algebra, given by
\[
A^{\alpha} = \big\{ a \in A \colon
\alpha_g (a) = a \mbox{ for all } g \in G \big\}.
\]
\end{ntn}
It is known that the Rokhlin property allows an
approximate homomorphism from the C*-algebra into the fixed point algebra
and 
approximate representability allows an approximate homomorphism from 
the crossed product into the C*-algebra in 
the more  general setting of finite groups 
(see \cite{BSV17, G19}).
  Now, in the following lemma, we show that strict approximate innerness, which is a weakening of approximate representability,
  suffices to provide 
 an approximate homomorphism from 
the crossed product into the C*-algebra in the setting of finite groups.
\begin{lem}\label{Prp_magic_a}
 Let $A$ be a unital \ca{} and let  $\alpha \colon  G \to \Aut(A)$ be an action of a 
finite  group $G$ on $A$ which is  strictly approximately inner. 
Then  for every finite
set $F \subset \CGAa$ and every $\ep > 0$, there exists a unital surjective linear map
$\psi \colon \CGAa \to A$ such that:
\begin{enumerate}
\item \label{Prp_magic_a.1}
$\| \psi \| \leq \card (G)$.
\item \label{Prp_magic_a.2}
$\psi (a) = a$ for all $a \in A$.
\item \label{Prp_magic_a.3}
$\| \psi (xy z^*) - \psi (x) \psi (y) \psi (z)^*\| < \ep$ for all $x, y, z \in F$.
\setcounter{TmpEnumi}{\value{enumi}}
\end{enumerate}
\end{lem}

\begin{proof}
Let $\ep > 0$ and let $F \subset \CGAa$ be a finite set.
Set 
\[
F_0= \Big\{ E_g (v), E_g (v)^* \colon g \in G \mbox{ and } v \in F \Big\},
\]
\[
M= \max \Big(1, \ \max_{a \in F_0}   \| a \|  \Big),
\qquad
\mbox{ and }
\qquad
\ep'= \frac{\ep}{2 M^2 \card (G)^3}.
\]  
Applying \Lem{Lem.s.a.i_hom} with $\ep'$ in place of
$\ep$ and $F_0$ as given, we get a \hm{} $g \mapsto z_g$
from $G$ to $U (A)$ such that, for all $g \in G$ and $a \in F_0$,
\begin{equation}\label{Ap_inner.a}
\| \af_g (a) - z_g a z_g^* \| < \ep'.
\end{equation}
Now define $\psi \colon \CGAa \to A$ by 
\[
\sum_{g \in G} a_g u_g \mapsto \sum_{g \in G } a_g z_g .
\]
Clearly $\psi$ is unital, linear, and surjective.
 To prove (\ref{Prp_magic_a.1}), let
 $\sum_{g\in G} a_g u_g$ be an arbitrary element in $\CGAa$. Then
\begin{align*}
\Bigg\| \psi \Bigg(\sum_{g\in G} a_g u_g \Bigg)\Bigg\|
= 
\Bigg\| \sum_{g \in G }  a_g z_g \Bigg\| 
\leq 
\sum_{g\in G} \|  a_g \| 
\leq 
\card (G) \Bigg\| \sum_{g\in G} a_g u_g \Bigg\|.
\end{align*}
Part (\ref{Prp_magic_a.2}) is immediate from the fact that $z_1=1_A$. 
To prove (\ref{Prp_magic_a.3}), let
\[
x=\sum_{g \in G} a_g u_g, 
\quad
y= \sum_{h\in G} b_h u_h, 
\quad
\mbox{and} 
\quad
z= \sum_{t\in G} c_t u_t
\]
 be arbitrary elements in $F$.
So 
$z^* =\sum_{t\in G} \alpha_{t^{-1}} (c^*_t) u_{t^{-1}}$. 
A short computation then shows that  
\begin{align}\label{Eq1100.190708}
\psi( xyz^* )  
=
 \sum_{g,h, t  \in G } a_g \alpha_{g} (b_h) \alpha_{ght^{-1}} (c^*_t) z_{gh t^{-1}}
\end{align}
and 
\begin{equation}\label{Eq11100.190708}
\psi(x) \psi(y) \psi(z)^*
=
  \sum_{g,h, t \in G } a_g z_g b_h z_h z^{*}_{t} c^{*}_{t}.
\end{equation}
Therefore, using (\ref{Eq1100.190708}) and (\ref{Eq11100.190708}) at the first step and
using (\ref{Ap_inner.a})  at the second step,
\begin{align*}
&\| \psi (xy z^*) - \psi (x) \psi(y) \psi(z)^*\| 
\\
&\hspace*{3em} {\mbox{}}\leq
\sum_{g,h, t  \in G} 
\|  a_g \| \cdot \| \alpha_{g} (b_h) -  z_g b_h z^*_g \| \cdot \|\alpha_{ght^{-1}} (c^{*}_t) z_{ght^{-1}} \|
\\&
\hspace*{6em} {\mbox{}}+
\sum_{g,h, t  \in G}
 \|  a_g z_g b_h z^*_g \| \cdot \| \alpha_{ght^{-1}} (c^{*}_{t})  -  z_{ght^{-1}} c^{*}_{t} z^*_{ght^{-1}} \| \cdot \| z_{ght^{-1}} \|
\\
&\hspace*{3em} {\mbox{}}
<
 \card(G)^3 M^2 \ep' + \card(G)^3 M^2 \ep' = 2 \card(G)^3 M^2 \ep' =\ep.
\end{align*}
This completes the proof.
\end{proof}
\begin{prp}\label{Pr_Magic_a}
Let $A$ be a unital \ca{} and let  $\alpha \colon  G \to \Aut(A)$ be an action of a 
finite  group $G$ on $A$ which is  strictly approximately inner. Assume $a, b \in A_+$.
Then $a \precsim_{A} b$ if and only if $a \precsim_{\CGAa} b$. 
\end{prp}
\begin{proof}
We only need to prove the backwards implication. Let $\ep> 0$. Let $a, b \in A_+$.
By Lemma~\ref{PhiB.Lem_18_4}(\ref{PhiB.Lem_18_4_11}), it suffices to show that
 $(a - \ep)_+ \precsim_{A} b$.
We use $a  \precsim_{C^*(G, A, \alpha)} b$  to find $v \in C^*(G, A, \alpha)$ such that 
\begin{equation}\label{Eq1.2019.08.28}
\left\| v b v^* - a \right\| < \tfrac{\ep}{2 \card (G)}.
\end{equation}
Set 
$F= \big\{b,  v, v^* \big\}$. 
Now, we apply Lemma~\ref{Prp_magic_a} with $\frac{\ep}{2}$ and $F$ as given to get a unital surjective linear map 
$\psi \colon \CGAa \to A $ such that:
\begin{enumerate}
\item \label{Pr_Magic_a.1}
$\| \psi \| \leq \card (G)$.
\item \label{Pr_Magic_a.2}
$\psi (c) = c$ for all $c \in A$.
\item \label{Pr_Magic_a.3}
$\left\| \psi (xy z^*) - \psi (x) \psi (y) \psi (z)^*\right\| < \frac{\ep}{2}$ for all $x, y, z \in F$.
\end{enumerate}
Using (\ref{Pr_Magic_a.2}) at the first step, 
using (\ref{Eq1.2019.08.28}) and (\ref{Pr_Magic_a.3}) at the second step, 
and using (\ref{Pr_Magic_a.1}) at the last step, we get
\begin{align}\label{Eq1.20200501}
\left\|a - \psi (v) b \psi (v)^* \right\|
&\leq
\left\| \psi \left(a\right) - \psi \left( v b v^*\right) \right\|
+
\left\| \psi \left( v b v^* \right) - \psi ( v ) \psi (b)  \psi ( v)^* \right\|
\\\notag
&<
  \frac{\ep \| \psi \|}{2 \card (G)}
+
\frac{\ep}{2} \leq \ep.
\end{align}
Therefore,  using (\ref{Eq1.20200501}) and 
Lemma~\ref{PhiB.Lem_18_4}(\ref{PhiB.Lem_18_4_10.a}) at the first step,
\[
(a - \ep)_+ \precsim_{A} \psi ( v ) b  \psi ( v)^*  \precsim_A b.
\]
This completes the proof.
\end{proof}
\begin{rmk}\label{Exten_Ap_in}
Let $A$ be a unital \ca, let $n \in \N$ and let  $\alpha \colon  G \to \Aut(A)$ be an action of a 
finite  group $G$ on $A$ which is  strictly approximately inner. 
Then it is easy to check that the action $\id_{M_n} \otimes \alpha \colon G \to \Aut (M_n (A))$
 is also strictly approximately inner.
\end{rmk}
\begin{thm}\label{WC_injectivity}
Let $A$ be a unital \ca{} and let $G$ be a finite group. Let $\alpha \colon G \to \Aut(A)$
be a strictly  approximately inner action.
Let $\iota \colon A \to \CGAa$ be the inclusion map.
Then:
\begin{enumerate}
\item\label{WC_injectivity_a}
The map $\W (\iota) \colon \W (A) \to \W (\CGAa)$
is an ordered semigroup isomorphism onto its range.
\item\label{WC_injectivity_b}
The map $\Cu (\iota) \colon \Cu (A) \to \Cu (\CGAa)$
is an ordered semigroup isomorphism onto its range.
\end{enumerate}
\end{thm}

\begin{proof}
We only need to show injectivity and order isomorphism in both parts.

Part (\ref{WC_injectivity_a}) is essentially immediate from Remark~\ref{Exten_Ap_in} and Proposition~\ref{Pr_Magic_a}.

We prove~(\ref{WC_injectivity_b}).
We must show that if $a, b \in ( \cK \otimes A)_{+}$
satisfy $a \precsim_{\CGAa} b$,
then $a \precsim_{A} b$.
Let $\ep > 0$.
By Lemma~\ref{PhiB.Lem_18_4}(\ref{PhiB.Lem_18_4_11}), it suffices to prove that
$(a - \ep)_{+} \precsim_{A} b$.
Applying Lemma~\ref{PhiB.Lem_18_4}(\ref{PhiB.Lem_18_4_11.c}) with $\frac{\ep}{3}$ in place of $\eta$, we choose $\dt > 0$ such that
\begin{equation}\label{Eq1_2019_04_14}
\left( a - \frac{\ep}{3} \right)_{+} \precsim_{\CGAa} (b - \dt)_{+}.
\end{equation}
Choose $m, n \in \N$,
$a_0 \in M_m (A)_+$, and $b_0 \in M_n (A)_+$
such that
\[
\bigg\| a_0 - \Big( a - \frac{\ep}{3} \Big)_{+} \bigg\|
   < \frac{\ep}{3}
\andeqn
\| b_0 - b \|
   < \frac{\dt}{2}.
\]
It follows from Lemma \ref{PhiB.Lem_18_4}(\ref{Item_9420_LgSb_1_6})
that
\begin{equation}\label{Eq_9417_FromIneq}
(a - \ep)_{+}
 \precsim_{A} \Big( a_0 - \frac{\ep}{3} \Big)_{+}
 \precsim_{A} \Big( a - \frac{\ep}{3} \Big)_{+}
\end{equation}
and
\begin{equation}\label{Eq_9417_FromIneq_2}
(b - \dt)_{+}
 \precsim_{A} \Big( b_0 - \frac{\dt}{2} \Big)_{+}
 \precsim_{A} b.
\end{equation}
Combining the second part of~(\ref{Eq_9417_FromIneq}),
the first part of~(\ref{Eq_9417_FromIneq_2}),
and~(\ref{Eq1_2019_04_14}),
we get
\begin{equation}\label{Eq_9417_InA}
\Big( a_0 - \frac{\ep}{3} \Big)_{+}
  \precsim_{\CGAa} \Big( b_0 - \frac{\dt}{2} \Big)_{+}.
\end{equation}
Since
$a_0, b_0 \in \bigcup_{l = 1}^{\I} M_l (A)$,
Part~(\ref{WC_injectivity_a})
implies that (\ref{Eq_9417_InA})
holds in $A$.
Combining this with
the first part of~(\ref{Eq_9417_FromIneq})
and the second part of~(\ref{Eq_9417_FromIneq_2}),
we get $(a - \ep)_{+} \precsim_{A} b$.
\end{proof}
We will give an example which satisfies the hypotheses 
of Theorem~\ref{WC_injectivity} in Section~\ref{Sec_Example} (see Corollary~\ref{C_9422_IsomsInExample}).
\begin{dfn}
 An ordered semigroup $S$ is said to be almost unperforated if for every
$x, y \in S$ and $n \in N$ such that $(n+1) x \leq n y$, then $x \leq y$. 
\end{dfn}
The following corollary is immediate from Theorem~\ref{WC_injectivity} and the fact that 
almost unperforation passes to sub-semigroups (with the induced order).
\begin{cor}
Let $A$ be a unital \ca{} and let $G$ be a finite group. Let $\alpha \colon G \to \Aut(A)$
be a strictly  approximately inner action.
If $\Cu (C^*(G, A, \alpha))$ is almost unperforated, then so is $\Cu (A)$.
\end{cor}
We give a lower bound and an upper bound for 
the radius of comparison of the crossed product by a strictly approximately inner action in the following theorem.
\begin{thm}\label{Thm_rc_In}
Let $A$ be a  stably finite unital \ca{} and let  $\alpha \colon  G \to \Aut(A)$ be a strictly approximately inner action of a 
finite  group $G$ on $A$. Then:
\begin{enumerate}
\item\label{Thm_rc_In.a}
$\rc (A) \leq \rc \big( \CGAa\big)$.
\item\label{Thm_rc_In.b}
If $C^*(G, A, \alpha)$ is simple, then $\rc (A) \leq \rc \big(C^*(G, A, \alpha)\big) \leq \rc (A^{\alpha})$.
\end{enumerate}
\end{thm}

\begin{proof}
To prove (\ref{Thm_rc_In.a}), assume $\rc (\CGAa) < \infty$.
Let $r \in [0, \I)$. Suppose that $\CGAa$ has $r$-comparison.
Let $a, b \in M_{\I} (A)_{+}$
satisfy 
\begin{equation}\label{Eq1.20200424}
d_{\rho} (a) + r < d_{\rho} (b)
\end{equation}
for all $\rho \in \QT (A)$.
Since every quasitrace on~$\CGAa$ restricts to a quasitrace on~$A$,
it follows from~(\ref{Eq1.20200424}) that 
\[
d_{\rho} (a) + r < d_{\rho} (b)
\]
for all $\rho \in \QT (\CGAa)$.
Now, since $\CGAa$ has $r$-comparison, we get $a \precsim_{\CGAa} b$.
Then, by \Thm{WC_injectivity}(\ref{WC_injectivity_a}), $a \precsim_{A} b$.
Therefore $\rc (A) \leq r$.
Taking the infimum over $r \in [0, \I)$
such that $\CGAa$ has $r$-comparison,
we get $\rc (A) \leq \rc (\CGAa)$.

To prove (\ref{Thm_rc_In.b}), assume the notation in Theorem~\ref{Ourcornertheorem} and let $p = \frac{1}{\card (G)} \sum_{g \in G} u_g$. 
Since $C^*(G, A, \alpha)$ is simple, $p$ is a  full projection in $C^*(G, A, \alpha)$.
Therefore, using (\ref{Thm_rc_In.a}) at the first step, 
using $0<\eta \leq 1$ at the second step, 
using Theorem~\ref{Ourcornertheorem} at the third step, 
and using Lemma~4.3(4) of \cite{AGP19} at the last step,
\[
\rc (A) \leq \rc \big(C^*(G, A, \alpha)\big) \leq \frac{1}{\eta} \cdot \rc \big(C^*(G, A, \alpha)\big) \leq 
\rc \big( p C^*(G, A, \alpha) p \big)  = \rc (A^\alpha),
\]
as desired.
\end{proof}
Let $p$ and $\eta$ be as in the proof of Theorem~\ref{Thm_rc_In}(\ref{Thm_rc_In.b}).
Suppose that 
$C^*(G, A, \alpha)$ is simple and
$\rc \big(C^*(G, A, \alpha)\big)\neq 0$. Since $\rho (p) <1$ for all $\rho \in C^*(G, A, \alpha)$, 
it follows that $\eta<1$. Therefore
$\rc \big(C^*(G, A, \alpha) \big) < \rc (A^{\alpha})$.  
In the light of Theorem~\ref{Thm_rc_In} and the results of \cite{AGP19}, it is shown in the following proposition that
actions of finite groups on many nonclassifiable C*-algebras cannot simultaneously have the Rokhlin property and be
 strictly approximately inner.
\begin{prp}\label{Pr.In.Ro}
Let $A$ be a stably finite unital \ca{} with $0<\rc (A) < \infty$. Then 
there is
no action of any nontrivial finite group on $A$ which both has the
Rokhlin property and is strictly approximately inner.
\end{prp}
\begin{proof}
Assume that $\alpha \colon G \to \Aut(A)$ is an action of a nontrivial group $G$ on $A$ which 
both has the Rokhlin property and is strictly approximately inner. Let $p = \frac{1}{\card (G)} \sum_{g \in G} u_g$.
Then, by Theorem~10.3.11 of \cite{GKPT18}, 
\begin{equation}
\label{Eq1.20200729}
\tau (p)= \frac{1}{\card (G)}
\end{equation}
for all $\tau \in C^*(G, A, \alpha)$. 
Therefore, using Theorem~\ref{Thm_rc_In}(\ref{Thm_rc_In.a}) at the first step, 
 using Lemma~4.3(6) of \cite{AGP19}, (\ref{Eq1.20200729}), and Theorem~\ref{Ourcornertheorem} at the second step,
 and using Theorem~4.2 of \cite{AGP19} at the third step,
\[
\rc(A) \leq \rc \left(C^*(G, A, \alpha)\right) = \frac{1}{\card(G)} \cdot \rc (A^{\alpha}) \leq \frac{1}{\card(G)} \cdot \rc (A).
\]
This contradicts $0<\rc(A) <\infty$.
\end{proof}
Proposition~\ref{Pr.In.Ro} might be true in the case that $\rc (A)= \infty$, although our method does not say anything about it.
Also, Proposition~\ref{Pr.In.Ro} is no longer valid if we do not have $\rc (A)>0$.
Here is the smallest counterexample.
\begin{exa}
\label{Count.Examp}
Let $\alpha$ be the action of $\mathbb{Z}/2 \mathbb{Z}$ on the $2^{\infty}$ UHF algebra 
generated by 
$\bigotimes_{n=1}^{\infty} \Ad 
\left(
\begin{matrix}
0&1\\
1&0
\end{matrix}
\right)
$
as in Example~10.3.6 of \cite{GKPT18} and let $F_2$ be the free group on two generators. 
 Then:
\begin{enumerate}
\item
 $\alpha$ has the Rokhlin property. 
\item
$\alpha$ is strictly approximately inner by Lemma~\ref{Lem_inlim_in_ap}.
\item
We can get such an action on a nonclassifiable simple unital C*-algebra by tensoring with the trivial action of
$\mathbb{Z}/2 \mathbb{Z}$ on 
the reduced group C*-algebra
$C^*_{\mathrm{r}} (F_2)$.
\end{enumerate} 
\end{exa}
\section{The radius of comparison and tracial strict approximate innerness}
\label{Sec_Tracial_Approx_Innner}
In this section, we essentially have a tracial analog of what we proved in Section~\ref{Sec_Approx_Innner}.

 The following definition is the tracial analog of Definition~\ref{D_9719_AppRep_Phi} and 
is  a  generalization of Definition~3.2 of~\cite{Ph11}
to nonabelian finite groups and not necessarily separable C*-algebras.
\begin{dfn}\label{D_9731_NonAbTrAppRep}
Let $A$ be an infinite-dimensional simple unital C*-algebra
and let $\af \colon G \to \Aut (A)$
be an action of a finite group $G$ on~$A$.
We say that $\af$ is
{\emph{tracially approximately representable}}
if for every finite set $F \S A$, every $\ep > 0$,
and every positive element $x \in A$ with $\| x \| = 1$,
there are a \pj{} $e \in A$
and unitaries $w_g \in \U (e A e)$ for $g \in G$ such that:
\begin{enumerate}
\item\label{D_9731_NonAbTrAppRep:1} 
$\| e a - a e \| < \ep$ for all $a \in F$.
\item\label{D_9731_NonAbTrAppRep:2} 
$\| \af_g (e a e) - w_g e a e w_g^* \| < \ep$
for all $a \in F$ and all $g \in G$.
\item\label{D_9731_NonAbTrAppRep:3} 
$\| w_g w_h - w_{g h} \| < \ep$ for all $g, h \in G$.
\item\label{D_9731_NonAbTrAppRep:4} 
$\| \af_g (w_h) - w_{g h g^{-1}} \| < \ep$ for all $g, h \in G$.
\item\label{D_9731_NonAbTrAppRep:5} 
$1 - e$ is \mvnt{} to a
\pj{} in the \hsa{} of $A$ generated by $x$.
\item\label{D_9731_NonAbTrAppRep:6} 
$\| e x e \| > 1 - \ep$.
\end{enumerate}
\end{dfn}
We  omit Condition (\ref{D_9731_NonAbTrAppRep:4}) in \Def{D_9731_NonAbTrAppRep} and 
 give the tracial analog of Definition~\ref{D_9719_StrictAppInn}.
\begin{dfn}\label{D_9731_TrStrAppInn}
Let $A$ be an infinite-dimensional simple  unital C*-algebra
and let $\af \colon G \to \Aut (A)$
be an action of a finite  group $G$ on $A$.
We say that $\af$ is
{\emph{tracially strictly approximately inner}}
if for every finite set $F \S A$, every $\ep > 0$,
and every positive element $x \in A$ with $\| x \| = 1$,
there are a \pj{} $e \in A$
and unitaries $z_g \in \U (e A e)$ for $g \in G$  such that:
\begin{enumerate}
\item\label{D_9731_TrStrAppInn:2.1} 
$\| e a - a e \| < \ep$ for all $a \in F$.
\item\label{D_9731_TrStrAppInn:2.2}
$\| \af_g (e a e) - z_g e a e z_g^* \| < \ep$
for all $a \in F$ and all $g \in G$.
\item \label{D_9731_TrStrAppInn:2.3}
$\| z_g z_h - z_{gh}\| < \ep$ for all $g, h \in G$.
\item\label{D_9731_TrStrAppInn:2.4} 
$1 - e$ is \mvnt{} to a
\pj{} in the \hsa{} of $A$ generated by $x$.
\item\label{D_9731_TrStrAppInn:2.5} 
$\| e x e \| > 1 - \ep$.
\end{enumerate}
\end{dfn}
In \Def{D_9731_TrStrAppInn},
it is clear that the algebra $A$ cannot be type~I.

To prove Theorem~\ref{Thm_magic_b}, we require invariance instead of approximate invariance
 in Definition~\ref{D_9731_TrStrAppInn}. We also require that $g \mapsto z_g$ be a homomorphism.
So, we provide a stronger version of Definition~\ref{D_9731_TrStrAppInn} for later applications in the following lemma.
\begin{lem}\label{L_Invariant_Pr}
Let $A$ be an infinite-dimensional simple  unital C*-algebra
and let $\af \colon G \to \Aut (A)$
be an action of a finite  group $G$ on $A$ which is tracially strictly approximately inner.
Then 
for every finite set $F \S A$, every $\ep > 0$,
and every positive element $x \in A$ with $\| x \| = 1$,
there are a \pj{} $e \in A^\af$
and a \hm{} $g \mapsto z_g$
from $G$ to $U (e A e)$ such that:
\begin{enumerate}
\item\label{D_9731_TrStrAppInn:1} 
$\| e a - a e \| < \ep$ for all $a \in F$.
\item\label{D_9731_TrStrAppInn:2}
$\| \af_g (e a e) - z_g e a e z_g^* \| < \ep$
for all $a \in F$ and all $g \in G$.
\item\label{D_9731_TrStrAppInn:4} 
$1 - e$ is \mvnt{} to a
\pj{} in the \hsa{} of $A$ generated by $x$.
\item\label{D_9731_TrStrAppInn:5} 
$\| e x e \| > 1 - \ep$.
\end{enumerate}
\end{lem}

\begin{proof}
The proof is essentially the same as the proof of Lemma 3.5 of \cite{Ph11}.
The differences  are as follows:
\begin{enumerate}
\item
Separability is not needed in the proof of that lemma.
\item
Semiprojectivity is used from $C^*(G)$ to $e A e$ not $(e A e)^\alpha$.
\item
Condition (4) in Lemma 3.5 of \cite{Ph11} is not needed in \Lem{L_Invariant_Pr}.
 So the assumption that $G$ is abelian   is not needed here.
\end{enumerate}
\end{proof}
Let $\alpha\colon G \to A$ be a tracially strictly approximately inner action of a finite group $G$
on an infinite-dimensional simple unital C*-algebra $A$. In the following theorem, we provide
a projection $e \in A^{\alpha}$ and
a unital surjective linear map  from $e C^*(G, A, \alpha) e$
to $e A e$ which is an approximate homomorphism on finite sets. Actually, it is the tracial analog of Lemma~\ref{Prp_magic_a}.
\begin{thm}\label{Thm_magic_b}
Let $A$ be an infinite-dimensional simple  unital \ca{} 
 and let  $\alpha \colon  G \to \Aut(A)$ be an action of a 
finite  group $G$ on $A$ which is tracially  strictly  approximately inner. 
Then  for every finite
set $F \subset \CGAa$, every $\ep > 0$, and every $a \in A_+$ with $\| a \|=1$, 
 there are a projection  $e \in A^{\alpha}$ and  a unital surjective  linear map
$\psi \colon e \CGAa e \to e A e$ such that:
\begin{enumerate}
\item \label{Thm_magic_b.1}
$\| \psi \| \leq \card (G)$.
\item \label{Thm_magic_b.2}
$\psi (e b e) = e b e$ for all $b \in A$.
\item \label{Thm_magic_b.3}
$\| \psi (e xy z^* e) - \psi (e x e) \psi (e y e) \psi (e z e)^*\| < \ep$ 
for all $x, y, z \in F$.
\item \label{Thm_magic_b.4}
$1-e \precsim_A a$.
\item \label{Thm_magic_b.5}
$\| e a e \| > 1-\ep$.
\setcounter{TmpEnumi}{\value{enumi}}
\end{enumerate}
\end{thm}

\begin{proof}
Let $\ep > 0$, let $a \in A_+$ with $\| a \|=1$, 
and let $F \subset \CGAa$ be a finite set.
Set 
\[
F_0= \Big\{ E_g (v), E_g (v)^* \colon g \in G \mbox{ and }  v \in F \Big\},
\]
\[
M= \max \Big( 1,\ \max_{b \in F_0} \| b \|  \Big),
\qquad
\mbox{ and }
\qquad
\ep'= \frac{\ep}{5 M^2 \card (G)^3}.
\]  
Applying \Lem{L_Invariant_Pr} with $\ep'$ and 
 $F_0$ as given,
we get a projection $e \in A^{\alpha}$ and
a \hm{} $g \mapsto z_g$
from $G$ to $U (e A e)$ such that:
\begin{enumerate}
\setcounter{enumi}{\value{TmpEnumi}}
\item \label{Tr_Ap_in.a}
$\| eb - be \| < \ep' $ for all $b \in F_0$.
\item \label{Tr_Ap_in.b}
$\| \alpha_g (ebe) - z_g ebe z^*_g\|< \ep'$ for all $b \in F_0$ and all $g \in G$.
\item \label{Tr_Ap_in.d}
$1-e \precsim_A a$.
\item \label{Tr_Ap_in.e}
$\| eae\| > 1 - \ep$.
\setcounter{TmpEnumi}{\value{enumi}}
\end{enumerate} 
So (\ref{Thm_magic_b.4}) and (\ref{Thm_magic_b.5}) are immediate. 

Now define $\psi \colon e \CGAa e \to e A e$ by 
\[
\sum_{g \in G} e a_g e u_g \mapsto \sum_{g \in G } e a_g e z_g .
\]
Clearly $\psi$ is unital, linear, and surjective. Also, Part~(\ref{Thm_magic_b.1}) and Part~(\ref{Thm_magic_b.2}) are immediate. 
 
 To prove (\ref{Thm_magic_b.3}), 
let
 $x=\sum_{g \in G} a_g u_g$, $y= \sum_{h\in G} b_h u_h$, and 
$z= \sum_{t\in G} c_t u_t$ be arbitrary elements in $F$.
So 
$z^* =\sum_{t\in G} \alpha_{t^{-1}} (c^*_t) u_{t^{-1}}$. 
Short computations then show that  
\begin{align}\label{Eq110.190708}
\psi( e xyz^* e)  
&=
\psi 
\left(
 \sum_{g,h,t  \in G} [e a_g \alpha_{g} (b_h) \alpha_{ght^{-1}} (c^*_t) e] u_{gh t^{-1}} 
 \right)
\\\notag
&=
 \sum_{g,h, t  \in G } e a_g \alpha_{g} (b_h) \alpha_{ght^{-1}} (c^*_t) e z_{gh t^{-1}}
\end{align}
and 
\begin{equation}\label{Eq111.190708}
\psi(exe) \psi(eye) \psi(eze)^*
=
  \sum_{g,h, t \in G } e a_g e z_g e b_h e z_{h t^{-1}} e c^{*}_{t} e.  
\end{equation}
Using (\ref{Tr_Ap_in.a}) and (\ref{Tr_Ap_in.b}) at the last step, 
we estimate, for all $b \in F_0$ and $g \in G$,
\begin{equation}\label{Eq1.20200326}
\| e \alpha_{g} (b) - z_g e b e z^*_g\| \leq \| \alpha_g (e b) - \alpha_g (ebe)\|
+
\| \alpha_g (ebe) - z_g e b e z^*_g \| < 2 \ep'.
\end{equation}
Therefore, using (\ref{Eq110.190708}) and (\ref{Eq111.190708}) at the first step and 
using (\ref{Eq1.20200326}) and (\ref{Tr_Ap_in.a}) at the second step,
\begin{align*}
&\| \psi (e xy z^* e) - \psi (e x e) \psi(e y e) \psi(e z e)^*\| 
\\
& \hspace*{2em} {\mbox{}} \leq
\sum_{g,h, t  \in G} 
\|e a_g -  e a_g e \| \cdot \| \alpha_{g} (b_h) \alpha_{ght^{-1}} (c^*_t) e z_{gh t^{-1}} \|
\\
&\hspace*{4em} {\mbox{}}+ 
\sum_{g,h, t  \in G} 
\|e a_g e \| \cdot \| e \alpha_{g} (b_h) - z_g e b_h e z^*_g \| \cdot \| \alpha_{ght^{-1}} (c^*_t) e z_{gh t^{-1}} \| 
\\
&  \hspace*{4em} {\mbox{}} + 
\sum_{g,h, t  \in G} 
\|e a_g e z_g e b_h e z^*_g \| 
\cdot 
\|  \alpha_{ght^{-1}} (c^*_t) e - z_{ght^{-1}} e c^*_t e z^*_{ght^{-1}} \|
 \cdot 
 \| e z_{gh t^{-1}} \|
\\
&  \hspace*{2em} {\mbox{}} <
 \card(G)^3 M^2 \ep' + 2\card(G)^3 M^2 \ep'+ 2\card(G)^3 M^2 \ep' = 5 \card(G)^3 M^2 \ep' =\ep.
\end{align*}
This completes the proof. 
\end{proof}
To prepare for Theorem~\ref{WC_plus_injectivity} and Theorem~\ref{Th_rc_Tr_inner}, the following lemmas are needed.
\begin{lem}\label{L_CR_Limit}
Let $A$ be an infinite-dimensional simple unital
 \ca{}, let $a, b \in A_+$, and 
 let  $\alpha \colon  G \to \Aut(A)$ be an action of a 
finite  group $G$ on $A$ which is tracially strictly   approximately inner. 
Assume $0$ is a limit point of $\spec (b)$.
Then  $a \precsim_{\CGAa} b$ if and only if 
  $a \precsim_{A} b$.
\end{lem}

\begin{proof}
We only need to prove the forwards implication. Let $\ep > 0$ and let  $\| a \|, \| b \| \leq 1$.
By Lemma~\ref{PhiB.Lem_18_4}(\ref{PhiB.Lem_18_4_11}), it suffices to show that
 $(a - \ep)_+ \precsim_{A} b$.
Applying Lemma~\ref{PhiB.Lem_18_4}(\ref{PhiB.Lem_18_4_11.c}) with $\frac{\ep}{2}$ in place of $\eta$, 
we choose $\dt > 0$ such that
$\left(a - \tfrac{\ep}{2} \right)_+  \precsim_{C^*(G, A, \alpha)} (b- \delta)_+$.
Set $a' = \big(a - \tfrac{\ep}{2}\big)_+$ and $b' = (b - \dt)_+$.
So, there exists $v \in C^*(G, A, \alpha)$ such that 
\begin{equation}
\left\| v  b'  v^* - a' \right\| < \frac{\ep}{4 \card (G)}.
\end{equation} 
Since $0$ is a limit point of $\spec(b)$, we can choose $\ld \in \spec (b) \cap (0, \dt)$.
Let $f \colon [0, \I) \to [0, 1]$
be a \cfn{}
such that $f (\ld) = 1$ and $\supp (f) \subseteq (0, \dt)$.
Then
\begin{equation}\label{Eq1.190711}
\| f (b) \| = 1,
\qquad
f (b) \perp b',
\andeqn
f (b) + b' \precsim_{A} b.
\end{equation}
Set 
$
F= \{ a', b',  v, v^* \}$
and 
$\ep'= \frac{\ep}{4}$.
Applying \Thm{Thm_magic_b} with $\ep'$ in place of $\ep$, $F$ as given, and  $f(b)$ in place of $a$,
we get a projection $e \in A^{\alpha}$ and  a unital surjective linear map 
$\psi \colon e \CGAa e \to e A e$ such that
\begin{enumerate}
\item \label{Thm_magic_b.111}
$\| \psi \| \leq \card (G)$.
\item \label{Thm_magic_b.222}
$\psi (e c e) = e c e$ for all $c \in A$.
\item \label{Thm_magic_b.333}
$\| \psi (e xy z^* e) - \psi (e x e) \psi (e y e) \psi (e z e)^*\| < \ep'$ 
for all $x, y, z \in F$.
\item \label{Thm_magic_b.444}
$1-e \precsim_A f(b)$.
\end{enumerate}
We use (\ref{Thm_magic_b.111}), (\ref{Thm_magic_b.222}), and (\ref{Thm_magic_b.333}) to get
\begin{align*}
\| e a' e  - \psi (e v e) e b' e \psi (e v e)^* \|
&\leq
\| \psi (e a' e ) - \psi ( e v b' v^* e) \|
\\
&\qquad +
\| \psi ( e v b' v^* e) - \psi (e v e ) \psi (e b' e)  \psi (e  v e)^* \|
\\
&<
 \frac{\ep \| \psi \|}{4 \card (G)}
+
\ep' < \frac{\ep}{2}.
\end{align*}
Then, by Lemma~\ref{PhiB.Lem_18_4}(\ref{PhiB.Lem_18_4_10.a}), 
\begin{equation}\label{Eq11.2019.08.29}
\big(e a' e - \tfrac{\ep}{2}\big)_+ \precsim_{A} \psi ( e v e ) e b'  e \psi (e v e)^*  \precsim_A b'.
\end{equation}
Since $\| e a e - e a' e \| < \tfrac{\ep}{2}$,
 it follows from Lemma~\ref{PhiB.Lem_18_4}(\ref{Item_9420_LgSb_1_6}) that 
\begin{equation}\label{Eq22.2019.08.29}
(e a e - \ep)_+ \precsim_A \big(e a' e - \tfrac{\ep}{2} \big)_+.
\end{equation}
Now, applying Lemma~12.1.5 of \cite{GKPT18} with $e$ in place of $g$, we get 
\begin{equation}\label{Eq33.2019.08.29}
(a - \ep)_+ \precsim_A ( e a e - \ep)_+ \oplus (1 - e).
\end{equation}
Therefore, using (\ref{Eq33.2019.08.29}) at the first step ,
using (\ref{Eq22.2019.08.29}) at the second step, 
using (\ref{Eq11.2019.08.29}) and (\ref{Thm_magic_b.444}) at the third step, 
and using the third part of (\ref{Eq1.190711}) at the last step,
\begin{align*}
(a - \ep)_+ \precsim_A ( e a e - \ep)_+ \oplus (1 - e)
\precsim_A
\big(e a' e - \tfrac{\ep}{2} \big)_+ \oplus (1 - e)
\precsim_A
b' \oplus f(b)
\precsim_A
b. 
\end{align*}
This completes the proof.
\end{proof}
Assume the notation in the proof of Lemma~\ref{L_CR_Limit}.
If $\alpha$ is strictly approximately inner, 
the assumption that $0$ is a limit point of $\spec (b)$ is not necessary (see Proposition~\ref{Pr_Magic_a}).
The problem occurs only with the tracial version of strict approximate innerness. 
In that case, if $0 \not\in \spec (b)$, then $v b v^*$ does not  cover all of $a$ in trace, and something besides
$(b - \ep)_{+}$ is needed to take care of the part of $a$ that is
missed. That is what $f(b)$ does.
\begin{ntn}
\label{NTN1.20200702}
For $n\in \N$, we abbreviate $1_{M_{n}}$ to $1_n$.
\end{ntn}
\begin{lem}\label{L_EXt_T.S.A.I}
Let $A$ be an infinite-dimensional simple   unital  \ca, let $n \in \N$, and 
let $\alpha \colon  G \to \Aut(A)$  be an action of a finite group $G$ on $A$ 
which is  tracially  strictly approximately inner.
Then $\id_{M_n} \otimes \alpha\ \colon G \to \Aut (M_n (A))$ is also tracially strictly  approximately inner. 
\end{lem}

\begin{proof}
To begin the proof,
for all $j, k \in \N$, we define  $E^{(jk)} \colon M_n (A) \to A$ by
$E^{(jk)} (c) = c_{j, k}$,
where $c=(c_{j, k})_{j,k =1}^{n} \in M_n (A)$.
Let $\ep>0$, let $F \subset M_n (A)$ be a  finite set,
 and let $a \in M_n (A)_+$ with $\| a \|=1$.
Choose $\dt \in (0, \ep)$ and 
choose $\lambda \in \big(\tfrac{1}{2}, 1\big)$ such that 
\begin{equation}\label{EQ1.20200708}
(1 - \dt ) \lambda ^2 > 1- \ep.
\end{equation}
 Set $b = ( a - \lambda)_+$.
Lemma~2.3 of \cite{HO13} 
 provides $x \in A_+ \setminus \{0\}$ such that 
\begin{equation}\label{Eq1.20200706}
1_{n} \otimes x \precsim_A b.
\end{equation} 
We may assume that $\| x \|=1$.
Using Lemma~2.7 of  \cite{AGP19}, we get $d \in M_n (A)$ such that 
\begin{equation}\label{Eq1.2019.08.29}
\| d a d^* - 1_{n} \otimes x  \| < \frac{\dt}{4} 
\quad
\mbox { and }
\quad
\| d \| \leq \| 1_{n} \otimes x  \|^{1/2} \lambda ^{-1/2} = \lambda ^{-1/2} < \lambda ^{-1} < 2.
\end{equation}
Set 
\[
F_0 = \left\{ E^{(jk)} (c)  \colon c  \in F \cup \{ d, d^* \} \mbox{ and } j,k= 1, 2, \ldots, n \right\}
\quad
\mbox{ and }
\quad
\ep' = \frac{\dt}{16 \, \card (G)^2}.
\]
Applying \Lem{L_Invariant_Pr} with $F_0$, $\ep'$, and $x$ as given, 
we find a projection $p \in A^{\alpha}$ 
and a \hm{} $g \mapsto z_g$
from $G$ to $\U (p A p)$ such that:
\begin{enumerate}
\item\label{Tr_Ap_in.a.1}
$\| py - yp \| < \ep' $ for all $y \in F_0$.
\item \label{Tr_Ap_in.a.2}
$\| \alpha_g (pyp) - z_g pyp z^*_g\|< \ep'$ for all $y \in F_0$ and all $g \in G$.
\item \label{Tr_Ap_in.a.4}
$1_A - p \precsim_{A} x$.
\item \label{Tr_Ap_in.a.5}
$\| pxp\| > 1 - \ep'$.
\setcounter{TmpEnumi}{\value{enumi}}
\end{enumerate}  
Set 
$e=  1_{n} \otimes p$
 and 
$f_g= 1_{n} \otimes z_g$. 
We now claim that the following hold:
\begin{enumerate}
\setcounter{enumi}{\value{TmpEnumi}}
\item \label{Tr_Ap_in.b.1}
$\| e c - c e \| < \ep $ for all $c \in F$.
\item \label{Tr_Ap_in.b.2}
$\big\| 
\big(\id_{M_n} \otimes \alpha_g\big) (e c e) 
- f_g e c e {f_g}^{*}
\big\|< \ep$
 for all $c \in F$ and all $g \in G$.
\item \label{Tr_Ap_in.b.3}
$ f_g f_h = f_{gh}$ for all $g, h \in G$.
\item \label{Tr_Ap_in.b.4}
$1_{M_n (A)} - e \precsim_{A} a$.
\item \label{Tr_Ap_in.b.5}
$\|e a e \| > 1 - \ep$.
\setcounter{TmpEnumi}{\value{enumi}}
\end{enumerate} 
To prove (\ref{Tr_Ap_in.b.1}), let $c \in F$.
Then, by (\ref{Tr_Ap_in.a.1}),
\begin{equation*}
\| e c - ce \| 
\leq
\sum_{j,k=1}^{n} \left\| p E^{(jk)} (c) -   E^{(jk)} (c) p \right\| 
<
\card (G)^2 \ep' < \dt < \ep.
\end{equation*}
To prove (\ref{Tr_Ap_in.b.2}), we use (\ref{Tr_Ap_in.a.2}) to estimate
\begin{align*}
\left\| 
\big(\id_{M_n} \otimes \alpha_g\big) (e c e) 
- f_g e c e f^*_g
\right\|
&\leq 
\sum_{j, k=1}^{n} 
\big\| \alpha_g \big(p E^{(jk)} (c ) p \big) 
- 
z_g p E^{(jk)} (c ) p z^*_g \big\|
\\&<
 \card (G)^2 \ep' < \ep.
\end{align*}
Part (\ref{Tr_Ap_in.b.3}) is immediate. 
To prove (\ref{Tr_Ap_in.b.4}), we use (\ref{Tr_Ap_in.a.4}) at the second step and use (\ref{Eq1.20200706}) 
at the third step to get
\begin{equation*}
1_{M_n (A)} - e = 1_{n} \otimes (1_A - p) 
\precsim_A 
 1_{n} \otimes x
\precsim_A
b \precsim_{A} a.
\end{equation*}
To prove (\ref{Tr_Ap_in.b.5}), we first use (\ref{Tr_Ap_in.a.5}) at the second step to get
\begin{equation}
\label{Eq5.20200622}
\| e (1_{n} \otimes x) e \| > \| p x p \| > 1 - \ep' > 1 - \frac{\dt}{2}.
\end{equation}
Using  the first part of (\ref{Eq1.2019.08.29}) at the second step,
 using (\ref{Tr_Ap_in.b.1}) and the second part of (\ref{Eq1.2019.08.29})  at the third step and the fourth step, 
and using the second part of  (\ref{Eq1.2019.08.29}) at the fifth step,
 we get
\begin{align*} 
\| e (1_{n} \otimes x) e \| 
&\leq
\|  e (1_{n} \otimes x) e - e d a d^* e \| + \| e d a d^* e \| 
\\\notag
&<
\frac{\dt}{4} + \| e d  -  d e \| \cdot \| a d^* e \| + \| d e a d^* e \|
\\\notag
&<
\frac{\dt}{4} + \card (G)^2 \ep' \lambda^{-1} + \| d e a \| 
\cdot \|d^* e - e d^*  \| + \| d\| \cdot \| e a  e\| \cdot \| d^* \|
\\\notag
&<
\frac{\dt}{4} + \card (G)^2 \ep' \lambda ^{-1} + \card (G)^2 \ep' \lambda^{-1} + \lambda ^{-2} \| e a  e\|
\\\notag
&<
\frac{\dt}{4} + 4 \card (G)^2 \ep'  + \lambda ^{-2} \| e a  e\|
\\\notag
&<
\frac{\dt}{4} + \frac{\dt}{4} + \lambda ^{-2} \| e a  e\|.
\end{align*}
This relation, (\ref{Eq5.20200622}), and (\ref{EQ1.20200708}) imply that
$
\| e a  e\| > (1 - \dt) \lambda ^2> 1-\ep.
$
\end{proof}
The following definition is taken from the discussion before Corollary 2.24 of~\cite{APT11}
and Definition~3.1 of~\cite{Ph14}
with slight changes in notation.
\begin{dfn}\label{D_9421_Pure}
Let $A$ be a \ca.
We define
\begin{itemize}
\item
$A_{++} = \big\{ a \in A_{+} \colon
  \mbox{there is no projection $q \in M_{\infty} (A)$
         such that $\langle a \rangle_A = \langle q \rangle_A$} \big\}$.
\item
$\Cu_{+} (A)
 = \big\{ \langle a \rangle_A \colon a \in (\cK \otimes A)_{++} \}$.
\item
$\W_{+} (A)
 = \Cu_{+} (A) \cap \W (A)$.
\end{itemize}
The elements of $A_{++}$ are called {\emph{purely positive}}.
\end{dfn}
\begin{rmk} \label{RMK_Sta_si}
If $A$ is a stably finite  simple unital C*-algebra, then
\[
(\cK \otimes A)_{++}
  = \big\{ a \in (\cK \otimes A)_{+} \colon
   \mbox{$0$ is a limit point of $\spec (a)$} \big\},
\]
and $\W_{+} (A) \cup \{ 0 \}$ and $\Cu_{+} (A) \cup \{ 0 \}$
are unital subsemigroups of $\W (A)$ and $\Cu (A)$.
\end{rmk}   
\begin{thm}\label{WC_plus_injectivity}
Let $A$ be a stably finite simple unital \ca{} which is not of type~I
and let $\alpha \colon G \to \Aut (A)$
be an action of a finite group $G$ on $A$
which is tracially strictly approximately inner.
Let $\iota \colon A \to \CGAa$ be the inclusion map.
Then:
\begin{enumerate}
\item\label{WC_plus_injectivity_a}
The map $\W (\iota) \colon \W (A) \to \W (\CGAa)$
induces an isomorphism of ordered semigroups
from $\W_{+} (A) \cup \{ 0 \}$
to its image in $\W (\CGAa)$.
\item\label{WC_plus_injectivity_b}
The map $\Cu (\iota) \colon \Cu (A) \to \Cu (\CGAa)$
induces an isomorphism of ordered semigroups
from $\Cu_{+} (A) \cup \{ 0 \}$
to its image in $\Cu (\CGAa)$.
\end{enumerate}
\end{thm}

\begin{proof}
Part~(\ref{WC_plus_injectivity_a}) is essentially immediate from
Lemma~\ref{L_CR_Limit}, Lemma~\ref{L_EXt_T.S.A.I}, and Remark~\ref{RMK_Sta_si}.

We prove~(\ref{WC_plus_injectivity_b}).
It suffices to prove
that if $a, b \in ( \cK \otimes A)_{++}$
satisfy $a \precsim_{\CGAa} b$,
then $a \precsim_{A} b$.
By Lemma~\ref{PhiB.Lem_18_4}(\ref{PhiB.Lem_18_4_11}), it is enough to show that, for every $\ep>0$, we have
$(a - \ep)_{+} \precsim_{A} b$.
So let $\ep > 0$.
By Lemma~\ref{PhiB.Lem_18_4}(\ref{PhiB.Lem_18_4_11.c}), there exists  $\dt > 0$ such that
\begin{equation}\label{Eq1_2019_04_14_wTRP}
\left( a - \frac{\ep}{3} \right)_{+} \precsim_{\CGAa} (b - \dt)_{+}.
\end{equation}
Since $0$ is a limit point of $\spec(b)$, we can choose 
 $\ld \in \spec (b) \cap \big( 0, \frac{\dt}{3} \big)$.
Let $f \colon [0, \I) \to [0, 1]$
be a \cfn{}
such that $f (\ld) = 1$
and $\supp (f) \subseteq \big( 0, \frac{\dt}{3} \big)$.
So, we have
\begin{equation}\label{Eq_9421_Prop_hb}
\| f (b) \| = 1,
\qquad
f (b) \perp \Big( b - \frac{\dt}{3} \Big)_{+},
\andeqn
f (b) + \Big( b - \frac{\dt}{3} \Big)_{+} \precsim_{A} b.
\end{equation}

Choose $n \in \N$ and $a_0, b_0, c_0 \in M_n (A)_{+}$
such that
\begin{equation}
\label{Eq2.20200702}
\bigg\| a_0 - \Big( a - \frac{\ep}{3} \Big)_{+} \bigg\|
   < \frac{\ep}{3},
\quad
\bigg\| b_0 - \Big( b - \frac{\dt}{3} \Big)_{+} \bigg\|
   < \frac{\dt}{3},
\quad
\mbox{and}
\quad
\| c_0 - f (b) \| < \frac{1}{3}.
\end{equation}
We use Lemma \ref{PhiB.Lem_18_4}(\ref{Item_9420_LgSb_1_6}) and (\ref{Eq2.20200702}) to get
\begin{equation}\label{Eq_9417_FromIneq_wTRP}
(a - \ep)_{+}
 \precsim_{A} \Big( a_0 - \frac{\ep}{3} \Big)_{+}
 \precsim_{A} \Big( a - \frac{\ep}{3} \Big)_{+}
\end{equation}
and
\begin{equation}\label{Eq_9417_FromIneq_2_wTRP}
(b - \dt)_{+}
 \precsim_{A} \Big( b_0 - \frac{\dt}{3} \Big)_{+}
 \precsim_{A} \Big( b - \frac{\dt}{3} \Big)_{+}.
\end{equation}
Set $d = \big( c_0 - \frac{1}{3} \big)_{+}$.
Clearly $\| d \| > \frac{1}{3}$. So $d \neq 0$.
Since $A$ is not of type~I and is simple,
$\overline{d M_n (A) d}$ is not of type~I and is simple.
Applying Lemma~2.1 of~\cite{Ph14} to $\overline{d M_n(A) d}$, we get $c \in \overline{d M_n(A) d}$ with $\spec (c)= [0, 1]$. 
Then, by the third part of (\ref{Eq2.20200702}),
\begin{equation}\label{Eq_9606_chb}
c \precsim_{A} f (b).
\end{equation}

At the first step
combining the second part of~(\ref{Eq_9417_FromIneq_wTRP}),
(\ref{Eq1_2019_04_14_wTRP}),
and the first part of~(\ref{Eq_9417_FromIneq_2_wTRP}),
we get
\begin{equation}\label{Eq_9417_InA_wTRP}
\Big( a_0 - \frac{\ep}{3} \Big)_{+}
  \precsim_{\CGAa} \Big( b_0 - \frac{\dt}{3} \Big)_{+}
  \precsim_{\CGAa} \Big( b_0 - \frac{\dt}{3} \Big)_{+} \oplus c.
\end{equation}
Since
$
a_0, \, \Big( b_0 - \frac{\dt}{3} \Big)_{+}, \, c
  \in \bigcup_{k = 1}^{\I} M_k (A),
$
it follows from Lemma~\ref{L_CR_Limit},
Part~(\ref{WC_plus_injectivity_a}), $\spec (c)=[0, 1]$, and
~(\ref{Eq_9417_InA_wTRP}) that
\begin{equation}\label{Eq_9423_InAaf_wTRP}
\Big( a_0 - \frac{\ep}{3} \Big)_{+}
  \precsim_{A} \Big( b_0 - \frac{\dt}{3} \Big)_{+} \oplus c.
\end{equation}
Therefore, using the first part of~(\ref{Eq_9417_FromIneq_wTRP}) at the first step,
using (\ref{Eq_9423_InAaf_wTRP}) at the second step,
using (\ref{Eq_9606_chb}) 
and the second part of~(\ref{Eq_9417_FromIneq_2_wTRP}) at the third step,
and using (\ref{Eq_9421_Prop_hb}) at the last step,
\[
(a - \ep)_{+}
 \precsim_{A} \Big( a_0 - \frac{\ep}{3} \Big)_{+}
 \precsim_{A} \Big( b_0 - \frac{\dt}{3} \Big)_{+} \oplus c
 \precsim_{A} \Big( b - \frac{\dt}{3} \Big)_{+} \oplus f (b)
 \precsim_{A} b.
\]
This completes the proof.
\end{proof}
The following corollary is immediate from Theorem~\ref{WC_plus_injectivity}, Proposition~2.8 of \cite{Th20},
Proposition~5.3.16 of \cite{APT18},
 and the fact that  almost unperforation passes to sub-semigroups (with the induced order).
\begin{cor}
Let $A$ be a stably finite simple unital \ca{} which is not of type~I
and let $\alpha \colon G \to \Aut (A)$
be an action of a finite group $G$ on $A$
which is tracially strictly approximately inner.
If $\Cu_+ \big(C^*(G, A, \alpha)\big) \cup \{ 0 \}$ is almost unperforated, then so is $\Cu (A)$.
\end{cor}
As we promised, we give a lower bound and an upper bound for 
the radius of comparison of the crossed product by a tracially strictly approximately inner action in the following theorem.
\begin{thm}\label{Th_rc_Tr_inner}
Let $A$ be an infinite-dimensional simple  unital stably finite  \ca{}
and let  $\alpha \colon  G \to \Aut(A)$ be a tracially strictly  approximately  inner action of a 
finite group $G$ on $A$. Then:
\begin{enumerate}
\item\label{Th_rc_Tr_inner.a}
$\rc (A) \leq \rc \big( \CGAa\big)$.
\item\label{Th_rc_Tr_inner.b}
If $C^*(G, A, \alpha)$ is simple, then $\rc (A) \leq \rc \big(C^*(G, A, \alpha)\big) \leq \rc (A^{\alpha})$.
\end{enumerate}
\end{thm}

\begin{proof}
To prove (\ref{Th_rc_Tr_inner.a}),
we use Proposition~\ref{rc_alg_def}. 
So, let $r \in [0, \I)$. Suppose that $\CGAa$ has $r$-comparison. Let $m, n \in \N$ satisfy $\frac{m}{n} > r$.
Let $k \in \N$ and let $a, b \in ( A \otimes M_k )_{+}$
with $\| a \| = \| b \| = 1$ satisfy
\begin{equation}\label{Eq2.20200424}
(n + 1) \langle a \rangle_{A}
    + m \langle 1 \rangle_{A}
 \leq n \langle b \rangle_{A}.
\end{equation}
By \Lem{L_EXt_T.S.A.I}, without loss of generality we can assume $k = 1$.
We have to prove that $a \precsim_{A} b$.
By \Lem{PhiB.Lem_18_4}(\ref{PhiB.Lem_18_4_11.b}), it suffices to show that
for every $\ep > 0$,
 we have $(a - \ep)_{+} \precsim_{A} b$.

So let $\ep > 0$.
\Wolog{} $\ep < \frac{1}{4}$.
Choose $l \in \N$ such that
\begin{equation}\label{Eq_9422_kmkn1}
\frac{ml}{ nl + 1} > r.
\end{equation}
Then, by (\ref{Eq2.20200424}),
\begin{equation}\label{Eq3.20200424}
(nl + 1) \langle a \rangle_{A}
 + ml \langle 1 \rangle_{A}
  \leq l(n + 1) \langle a \rangle_{A}
     + m l \langle 1 \rangle_{A}
  \leq n l \langle b \rangle_{A}.
\end{equation}
We define the following elements:
\begin{itemize}
\item
$x \in M_{\infty} (A)_{+}$
is the direct sum of $nl + 1$ copies of~$a$.
\item
$y \in M_{\infty} (A)_{+}$
is the direct sum of $nl$ copies of~$b$.
\item
$p \in M_{\infty} (A)_{+}$
is the direct sum of $ml$ copies of~$1_A$.
\end{itemize}

By (\ref{Eq3.20200424}), we have $x \oplus p \precsim_{A} y$.
Then, by \Lem{PhiB.Lem_18_4}(\ref{PhiB.Lem_18_4_11.c}), there exists $\dt > 0$ such that
$\big( x \oplus p - \ep \big)_{+}
 \precsim_{A} (y - \dt)_{+}$.
Using $\ep < \frac{1}{4}$ at the last step , we get
 \[
( x \oplus p - \ep )_{+}
  = ( x - \ep )_{+}
     \oplus ( p - \ep )_{+}
  \sim_{A} ( x - \ep )_{+} \oplus p,
\]
and therefore
\begin{equation}\label{Eq_9606_Star}
(nl + 1)
    \langle ( a - \ep )_{+} \rangle_{A}
              + ml \langle 1 \rangle_{A}
  \leq nl \langle (b - \dt)_{+} \rangle_{A}.
\end{equation}
Applying Lemma~2.7 of~\cite{Ph14} with $\dt$ in place of $\ep$ and with $nl$ in place of $n$, we get
elements $c \in {A_+}$ and $d \in {A_+} \setminus \{ 0 \}$
such that
\begin{equation}\label{Eq_5513_cyb}
nl \langle (b - \dt)_{+} \rangle_{A}
   \leq (nl + 1) \langle c \rangle_{A}
\andeqn
\langle c \rangle_{A} + \langle d \rangle_{A}
  \leq \langle b \rangle_{A}.
\end{equation}
Then, by (\ref{Eq_9606_Star}) and~(\ref{Eq_5513_cyb}),
\[
(nl + 1)
  \langle (a - \ep)_{+} \rangle_{A}
              +  ml \langle 1 \rangle_{A}
  \leq (nl + 1) \langle c \rangle_{A}.
\]
This relation also holds in $\W (\CGAa)$.
For all $\rho \in \QT \big(\CGAa \big)$,
apply $d_{\rho}$ and divide by $nl + 1$
to get
\begin{equation}
\label{Eq3.20200702}
d_{\rho} ( (a - \ep)_{+} ) + \frac{ml}{nl + 1}
\leq d_{\rho} (c).
\end{equation}
We now use (\ref{Eq_9422_kmkn1}) and (\ref{Eq3.20200702}) to get 
\begin{equation}
\label{Eq4.20200702}
(a - \ep)_{+} \precsim_{\CGAa} c.
\end{equation}
Since $A$ is simple and is  not of type~I,
$\overline{dAd}$ also is simple and is not of type~I. Now, applying Lemma~2.1 of~\cite{Ph14} to $\overline{dAd}$,
we can choose $d_0 \in \overline{dAd}$ with $\spec (d_0)= [0, 1]$.
Replacing $d$ with this element,
we may assume that $\spec (d)= [0, 1]$.
Therefore,
using Lemma~\ref{L_CR_Limit}
with $c \oplus d$ in place of $b$ and (\ref{Eq4.20200702}) at the first step
and using the second part of (\ref{Eq_5513_cyb}) at the second step,
\[
(a - \ep)_{+} \precsim_{A} c \oplus d
  \precsim_{A} b.
\]
Therefore $\rc (A) \leq r$.
Taking the infimum over $r \in [0, \I)$
such that $\CGAa$ has $r$-comparison,
we get $\rc (A) \leq \rc (\CGAa)$.  
This completes the proof of (\ref{Th_rc_Tr_inner.a}).

The proof of (\ref{Th_rc_Tr_inner.b}) is as same as the proof of Theorem~\ref{Thm_rc_In}(\ref{Thm_rc_In.b}),
 except that we now apply
(\ref{Th_rc_Tr_inner.a}) instead of Theorem~\ref{Thm_rc_In}(\ref{Thm_rc_In.a}).
\end{proof}
In the following proposition, we show that
actions of finite groups on many  nonclassifiable simple  unital C*-algebras cannot simultaneously
 have the weak tracial Rokhlin property (Definition~3.2 of \cite{AGP19}) and be tracially strictly approximately inner.
Note that we can deduce Proposition~\ref{Pr.In.Ro} from the following proposition
in the setting of infinite-dimensional simple C*-algebras.
\begin{prp}
\label{Pr.Tracial.In.Ro}
Let $A$ be an infinite-dimensional simple  unital stably finite \ca{} with $0<\rc (A)< \infty$.
Then there is no action of any nontrivial finite group on $A$ which both has 
the weak tracial Rokhlin property and is tracially strictly approximately inner.
\end{prp}
\begin{proof}
The proof is essentially the same as the proof of Proposition~\ref{Pr.In.Ro},
 except that we now
apply Theorem~4.6 of \cite{AGP19} and Theorem~\ref{Th_rc_Tr_inner}(\ref{Th_rc_Tr_inner.a}).
\end{proof}
The assumption that $\rc(A)>0$ is necessary for Proposition~\ref{Pr.Tracial.In.Ro}
(see Example~\ref{Count.Examp} or  Example~\ref{Ex_1_Z_2}). Proposition~\ref{Pr.Tracial.In.Ro}
might also hold for $\rc (A) = \infty$, but the proof does not cover
this case.

In the following corollary, we show that if actions of finite groups on  simple  unital C*-algebras simultaneously
 have the weak tracial Rokhlin property and is tracially strictly approximately inner, then the C*-algebras have to be classifiable.
\begin{cor}\label{Cor.R.I.almost}
Let $A$ be an infinite-dimensional simple  unital stably finite \ca{} with $\rc (A)< \infty$ and 
let  $\alpha \colon  G \to \Aut(A)$ be an action of a 
(nontrivial) finite  group $G$ on $A$.
If $\alpha$ has the weak tracial Rokhlin property and is tracially strictly approximately inner. Then
$\Cu (A)$ is almost unperforated.
\end{cor}
\begin{proof}
We use Theorem~\ref{Th_rc_Tr_inner}(\ref{Th_rc_Tr_inner.a}) and Theorem~4.6 of \cite{AGP19} to get
\[
\rc(A) \leq \rc (C^*(G, A, \alpha)) \leq \frac{1}{\card (G)} \cdot \rc(A).
\]
This relation implies that $\rc (A) =0$. Then, by Proposition~6.4 of \cite{Tom06}, $\Cu (A)$ is almost unperforated.
\end{proof}
It is plausible to hope that Corollary~\ref{Cor.R.I.almost} might be true 
if almost unperforation of $\Cu (A)$ is replaced by tracial $\mathcal{Z}$-stability of $A$.
%
\section{Examples}
\label{Sec_Example}
In this section, we give examples which satisfy the hypothesis of Theorem~\ref{Th_rc_Tr_inner}.
In Example~\ref{Ex_1_Z_2} and Example~\ref{Ex_2_Z_2}, the  groups are abelian and the C*-algebras are classifiable.
In the case that $G$ is a finite group (not necessarily abelian)
 and $A$ is a nonclassifiable C*-algebra, we show that
  for any finite group (not necessarily abelian) and any $r \in \big(0, \frac{1}{\card (G)}\big)$,
there exist a nonclassifiable C*-algebra $A$ and a
strictly approximately inner action $\alpha \colon G \to \Aut (A)$, which is also pointwise outer, such that 
$\rc (A)= \rc \big( C^*(G, A, \alpha)\big)=r$.

The following example is taken from Example~2.8 of \cite{Ph15}.
\begin{exa}\label{Ex_1_Z_2}
Let $D$ be the $2^{\infty}$ UHF algebra and let  $\alpha$ be the automorphism of order $2$ given by
\[
\alpha= \bigotimes_{n=1}^{\infty} \Ad \big(1_{2^{n-1} +1} \oplus (-1_{2^{n-1} -1})\big).
\]
\begin{enumerate}
\item
The action $\alpha$ is approximately representable
and therefore it is strictly approximately inner.
\item
The dual action $\widehat{\alpha}$ has the strict Rokhlin property.
\item
The dual action $\widehat{\alpha}$ is not strictly approximately inner 
since $\widehat{\alpha}$ is nontrivial on $K_0 \big( C^* (\mathbb{Z} /2\mathbb{Z}, D, \alpha) \big)$.
\item
The dual action $\widehat{\alpha}$ is tracially strictly approximately representable.
\item
$C^* (\mathbb{Z} /2\mathbb{Z}, D, \alpha)$ 
is a  simple unital AF algebra with a unique tracial state by Example 2.8 of \cite{Ph15}.
\item
We have 
$\rc (D)= \rc \big( C^* (\mathbb{Z} /2\mathbb{Z}, D, \alpha) \big)=\rc (D^{\alpha})=0$.
\end{enumerate}
\end{exa}
\begin{exa}\label{Ex_2_Z_2}
Let $A$ and $\alpha \in \Aut(A)$ be as in Example 4.1 of \cite{Ph15}.
\begin{enumerate}
\item
The action of $\mathbb{Z} /2\mathbb{Z}$ generated by $\alpha$ 
is tracially approximately representable but not approximately representable.
Therefore it is tracially strictly approximately inner.
\item
$A$ is a simple unital AF-algebra with a unique tracial state.
\item
$C^* (\mathbb{Z} /2\mathbb{Z}, A, \alpha)$ 
is a simple unital AH algebra with no dimension growth, tracial
rank zero, and a unique tracial state.
\item
$\rc (A)= \rc \big( C^* (\mathbb{Z} /2\mathbb{Z}, A, \alpha) \big)=\rc (A^{\alpha})=0$.
\end{enumerate}
\end{exa}
From now on, we concentrate on the setting of nonclassifable C*-algebras and nonabelian finite groups.
\begin{ntn}
Let $G$ be a topological group,
let $H_1$ and $H_2$ be Hilbert spaces,
and let $w_1 \colon G \to \U (H_1)$ and $w_2 \colon G \to \U (H_2)$
be unitary representations.
 We denote by $w_1 \otimes w_2$ the unitary
representation on $H_1 \otimes H_2$ such that 
\[
(w_1 \otimes w_2) (g) =
w_1 (g) \otimes w_2 (g)
\]
 for all $g \in G$.
\end{ntn}

\begin{ntn}\label{D_3407_DRegRep}
Let $G$ be a discrete group.
The {\emph{left regular representation}}
of $G$ is the representation  $v \colon G \to \U (l^2 (G))$
given by
$(v (g) \xi) (h) = \xi (g^{-1} h)$
for all $g, h \in G$ and all $\xi \in l^2 (G)$.
\end{ntn}

\begin{lem}[Fell Absorption Principle]
\label{Left_R_absorb}
Let $G$ be a discrete group, let $v$ be the left regular representation of $G$ on $l^2 (G)$,
and let $w$ be a unitary representation of $G$ on a Hilbert space $H$.
Then 
$v \otimes w$ is unitarily equivalent
to the tensor product of $v$ and the trivial representation
of $G$ on~$H$.
\end{lem}

\begin{proof}
It was shown in the proof of Theorem~9.5.7 in \cite{GKPT18}.
\end{proof}

The following lemma is Exercise 32 in Section 1.5 of  \cite{FR99}.
\begin{lem}\label{Folland_Exr}
Let  $(\lambda_j)_{j=1}^{\infty}$ be a sequence in $(0, 1)$. Then
 $\prod_{j=1}^{\infty} (1 - \lambda_j)>0$
if and only if $\sum_{j=1}^{\infty} \lambda_j < \infty$.

\end{lem}
\begin{lem}\label{Special_seq}
Let $r \in (0, 1)$ and let $m \in \N$. Then there exists a nondecreasing sequence 
$(d(n))_{n \in \N}$ in $\N$ such that 
\[
\lim_{n \to \infty} d(n)= \infty
\qquad
\mbox{and}
\qquad
\prod_{n= 1}^{\infty} \left( 1 -  \frac{m}{d(n)+m} \right) = r.
\]
\end{lem}
\begin{proof}
We construct the sequence by induction on $n$. 
To begin,
 set $r_0=1$. We note that $r_0>r$. Then define
\[
 d(1)= \min \left( \left\{k \in \N \colon 1 - \frac{m}{k+ m} > r \right\} \right)
\quad
\mbox{and}
\quad 
r_1= 1- \frac{m}{d(1)+ m}.
\]
Obviously, the  set in the definition of $d(1)$
  is nonempty and  $r_0 \geq r_1>r$.
 Then, for $n\geq 2$, set 
\begin{equation*}
d(n)= \min \left( 
\left\{
k \in \N \colon  1 -  \frac{m}{k+ m} > \frac{r}{r_{n-1}} 
\right\} 
\right) 
\hspace*{.1em}
\mbox{ and }
\hspace*{.1em} 
r_{n}= 
\prod_{k=1}^{n} \left(1- \frac{m}{d(k)+ m}\right).
\end{equation*}
Now we claim that
\begin{enumerate}
\item\label{Claim_1.a}
$r_{n} \geq r_{n+1}> r$ for all $n \in \N$.
\item\label{Claim_1.b}
$d(n) \leq d(n+1)$ for all $n \in \N$.
\item\label{Claim_1.c}
$\lim_{n\to \infty} r_n = r$.
\end{enumerate}
Part (\ref{Claim_1.a}) is immediate. We prove (\ref{Claim_1.b}). Since 
$\frac{r}{r_{n}} > \frac{r}{r_{n-1}}$, it follows that
\[
\left\{
k \in \N \colon  1 -  \frac{m}{k+ m} > \frac{r}{r_{n}} 
\right\}
\subseteq
\left\{
k \in \N \colon  1 -  \frac{m}{k+ m} > \frac{r}{r_{n-1}} 
\right\} 
\]
Using this, we get
\begin{align*}
d(n) 
&= 
\min \left(
\left\{
k \in \N \colon  1 -  \frac{m}{k+ m} > \frac{r}{r_{n-1}} 
\right\}
\right)
\\&
\leq
\min \left(
\left\{
k \in \N \colon  1 -  \frac{m}{k+ m} > \frac{r}{r_{n}} 
\right\}
\right)
 =d(n+1).
\end{align*}

We prove (\ref{Claim_1.c}). By (\ref{Claim_1.a}), it is immediate that $(r_n)_{n \in \N}$ converges to an element $r'$.
It suffices to show that $r'=r$.
 Assume $r'>r$. 
 Since $(r_n)_{n \in \N}$ converges to $r'>0$, 
 it follows from Lemma~\ref{Folland_Exr} that $\lim_{n \to \infty} d(n)= \infty$.
Then we choose $k$ large enough such that
\[
\left(1- \frac{m}{d(k) - 1+ m}\right)r'> r.
\]
Using this and using $r_{k-1}\geq r'$, we get 
\[
 \left(1- \frac{m}{d(k) - 1+ m}\right)> \frac{r}{r_{k-1}}.
\]
This contradicts the definition of $d(k)$.
\end{proof}
The following construction is motivated by Construction~6.1 of \cite{AGP19} and Construction~1.1 of \cite{HP19}.  
In \cite{AGP19}, the essential point was that the group is $\Z/2 \Z$ and the action has the Rokhlin property.
 Here, our construction is in a sense dual to the
the example as in Section~6 of \cite{AGP19}, but in the dual version it is easier to see how
to generalize to groups other than $\Z/2 \Z$  (especially to groups
which are not cyclic) and enables the use of nonabelian groups.
This modification will lead to further difficulties. 
\begin{ctn}\label{Ctn}
We define the following objects:
\begin{enumerate}
\item\label{Ctn_1}
Let $G$ be a (nontrivial) finite group  and 
let $z \colon G \to \U (l^2 (G))$ be the left regular  representation. Set $\nu=\card (G)$.
Since $G$ is finite, we have $B (l^2 (G)) \cong M_{\card (G)}$.
\item\label{Ctn_2}
Let $\eta\in \left(0, \frac{1}{\nu}\right)$. So $\nu \eta\in (0, 1)$.
Applying  Lemma~\ref{Special_seq} with $\nu \eta$ in place of $r$ and $\nu$ in palce of $m$,
we get a nondecreasing sequence
 $(d(n))_{n \in \N}$ in $\N$ such that 
\begin{equation*}
 \lim_{n \to \infty} d(n)= \infty
\qquad
\mbox{ and }
\qquad
 \prod_{n = 1}^{\infty} \left( 1 -  \frac{\nu}{d(n)+\nu} \right) = \nu \eta.
\end{equation*}
\item\label{Ctn_3}
For $n \in \N$, define
\begin{itemize}
\item
$l(n) =  d(n) + \nu$.
\item
$s(0)=1$ and $s (n) = \prod_{k = 1}^n d(k)$.
\item
$r(0)=1$ and $r (n) = \prod_{k = 1}^n l(k)$.
\item
$u (0)=1$ and 
$u (n)
 = \frac{s (n)}{r (n)}
 = \prod_{k = 1}^n \left(1 - \frac{\nu}{ d(k) + \nu} \right)$.
\end{itemize}
\item\label{Ctn_4}
It our computations, we will need
the following identifications of tensor products of matrix algebras:
\begin{enumerate}
\item\label{Ctn_4.ITN1}
For $n\in \N$, define an  isomorphism $\theta \colon M_{\nu} \otimes M_{\nu} \to M_{\nu^2}$ by 
\[
a \otimes [b_{jk}]_{j,k =1}^{\nu}  \mapsto \left[a b_{jk} \right]_{j,k =1}^{\nu}.
\]
\item\label{Ctn_4.ITN2}
For $n \in \N$, define an  isomorphism $\varphi_{n} \colon M_{\nu^2} \otimes M_{n} \to M_{\nu^2 n}$ by 
\[
a \otimes [b_{jk}]_{j, k =1}^{n}  \mapsto \left[a b_{jk} \right]_{j,k =1}^{n}.
\]
\item\label{Ctn_4.ITN3}
For $n \in \N$, define an  isomorphism $\psi_{n} \colon M_{\nu} \otimes M_{\nu n} \to M_{\nu^2 n}$ by 
\[
a  \otimes  [b_{jk}]_{j,k =1}^{\nu n}
\mapsto
 \left[a b_{jk} \right]_{j, k =1}^{\nu n}.
\]
\item\label{Ctn_4.ITN4}
For $n \in \N$, define an  isomorphism $\sigma_{n} \colon M_{\nu} \otimes M_{n} \to M_{\nu n}$ by 
\[
a \otimes [b_{jk}]_{j,k =1}^{n} 
\mapsto
 \left[a b_{jk} \right]_{j, k =1}^{n}.
\]
\end{enumerate} 
\item\label{Ctn_5}
By Lemma~\ref{Left_R_absorb}, there is $w \in \U (M_{\nu} \otimes M_{\nu})$ such that, for all $g \in G$,
\begin{equation*}
w (z_{g} \otimes z_{g})w^* =  z_{g} \otimes 1_{\nu}.
\end{equation*}
\item\label{Ctn_6}
For $n \in \Nz$,
define a compact space by
$X_{n} = (S^{2})^{s (n)}$.
Then the covering dimension of $X_n$
is $\dim ( X_{n} ) = 2 s (n)$.
\item\label{Ctn_7}
For $n \in \Nz$
and $j = 1, 2, \ldots, d (n + 1)$,
let $P^{(n)}_j \colon X_{n + 1} \to X_n$
be the $j$~coordinate projection.
\item\label{Ctn_8}
Choose points $x_m \in X_m$
for $m \in \N$ such that
for all $n \in \Nz$,
the set
\begin{align*}
&
\Bigl\{ \bigl( P^{(n)}_{\nu_{1}} \circ P^{(n + 1)}_{\nu_{2}}
      \circ \cdots \circ P^{(m - 1)}_{\nu_{m - n}} \bigr) (x_m) \colon
\\
& \hspace*{1em} {\mbox{}}
       {\mbox{$m = n + 1, \, n + 2, \, \ldots$
       and $\nu_j = 1, 2, \ldots, d (n + j)$
       for $j = 1, 2, \ldots, m - n$}} \Bigr\}
\end{align*}
is dense in $X_n$.

\item\label{Ctn_9}
For $n \in \Nz$, define
\[
A_{n}
 = C \left(X_{n}, M_{\nu r (n)}\right).
\]
\item\label{Ctn_10-}
For $n \in \Nz$, $f \in  C \left(X_{n}, M_{\nu r (n)}\right)$, and $x \in X_n$, define
\[
c_f (x)
= 
\varphi_{r(n)} \left(\theta(w) \otimes 1_{r(n)} \right)
\psi_{r(n)} \left(1_{\nu} \otimes f(x)\right)
\varphi_{r(n)} \left(\theta(w^*) \otimes 1_{r(n)} \right).
\]
\item\label{Ctn_10}
For $n \in \Nz$, define
\[
\Gamma_{n + 1, \, n} \colon A_n  \to A_{n+1}
\] by
\begin{equation*}
[\Gamma_{n + 1, \, n} (f)](x) =
\left( \begin{matrix}
    \left(f \circ P^{(n)}_{1}\right) (x) & & & & 0 \\
    & &  \ddots      & &     \\
    & & & \left(f \circ P^{(n)}_{d (n + 1)} \right)(x) \\
 0  & & & &   
 c_f (x_n) 
 \end{matrix} \right).
\end{equation*}
for $f \in C \left(X_{n}, M_{\nu r(n)}\right)$, $x \in X_{n+1}$, and $x_n \in X_n$.
For $m, n \in \Nz$ with $m \leq n$,
now define
\[
\Gamma_{n, m}
= \Gamma_{n, n - 1} \circ \Gamma_{n - 1, \, n - 2} \circ \cdots
          \circ \Gamma_{m + 1, m}
   \colon A_m \to A_n.
\]
\item\label{Ctn_11}
Define
\[
A = \dirlim \big( A_n, \, (\Gamma_{m, \, n})_{m \geq n} \big).
\]
For $n \in \Nz$, it is clear that $\Gamma_{n + 1, \, n}$
is an injective unital homomorphism.
Let $\Gamma_{\infty, n} \colon A_n \to A$
be the standard map associated with the direct limit.
\item\label{Ctn_12}
For $n \in \Nz$, define
$\alpha^{(n)} \colon G \to \Aut (A_n)$  by
\begin{align*}
\left[\alpha^{(n)}_g (f)\right](x) 
&= 
\Ad\left(\sigma_{r(n)} (z_g\otimes 1_{r(n)})\right) (f(x))
\\&
=
\sigma_{r(n)} \left(z_g\otimes 1_{r(n)}\right) f(x) \ \sigma_{r(n)} \left(z^*_g\otimes 1_{r(n)}\right)
\end{align*}
for all $g \in G$, $f\in C (X_{n}, M_{\nu r(n)})$, and  $x \in X_n$.
We then have the following diagram:
\begin{equation}\label{Eq_9411_LimDiag}
\begin{CD}
A_0
@>{\Gamma_{1, \, 0}}>>
A_1
 @>{\Gamma_{2, \, 1}}>>
A_2
@>{\Gamma_{3, \, 2}}>>
A_3
@>{}>>
 \cdots   \\
@V{\alpha^{(0)}_{g}}VV  @V{\alpha^{(1)}_{g}}VV
@V{\alpha^{(2)}_{g}}VV     @V{\alpha^{(3)}_{g}}VV      \\
A_0
@>{\Gamma_{1, \, 0}}>>
A_1
@>{\Gamma_{2, \, 1}}>>
A_2
@>{\Gamma_{3, \, 2}}>>
A_3
@>{}>>
 \cdots.
\end{CD}
\end{equation}
\end{enumerate}
\end{ctn}
%
\begin{lem}\label{Lem_Identification}
Assume the notation and choices in
Construction~\ref{Ctn}(\ref{Ctn_4.ITN1}),
Construction~\ref{Ctn}(\ref{Ctn_4.ITN2}), Construction~\ref{Ctn}(\ref{Ctn_4.ITN3}),
and Construction~\ref{Ctn}(\ref{Ctn_4.ITN4}).
Let $n \in \N$, let $a, b\in M_{\nu}$, and let $c\in M_{n}$.
Then 
\[
\psi_n \big(a \otimes \sigma_n \left(b \otimes c\right)\big)
= 
\varphi_n \big(\theta \left(a \otimes b\right) \otimes c\big).
\]
\end{lem}
\begin{proof}
It follows from explicit computations.
\end{proof}

\begin{lem}\label{L_Kap_20200325}
Assume the notation and choices in
Construction~\ref{Ctn}(\ref{Ctn_2}) and Construction~\ref{Ctn}(\ref{Ctn_3}).
Then $(u (n))_{n \in \Nz}$ is  decreasing
and $\lim_{n \to \infty} u(n)=\nu \eta$.
\end{lem}
\begin{proof}
The proof is immediate from Construction~\ref{Ctn}(\ref{Ctn_2}) and Construction~\ref{Ctn}(\ref{Ctn_3}).
\end{proof}
\begin{lem}\label{Approx_Commutnat}
Let $\ep>0$ and  let $n \in \N$. Let $b$ be a normal element in $M_n$ with $\| b \| = 1$.
Suppose that $\| b c - c b\| < \ep$ for all $c \in M_n$ with $\| c \|=1$.
Then there is $\xi \in \mathbb{T}$ such that  $\| b - \xi 1\| < \ep$.
\end{lem}
\begin{proof}
Since $a$ is normal, it follows from the Spectral Theorem that 
\begin{equation} \label{Eq1.20200324}
a=\sum_{j=1}^{n} \xi_j p_j, 
\end{equation}
where $p_1, \ldots,p_n$ are mutually orthogonal rank one projections in $M_n$, 
$\sum_{j=1}^{n} p_j=1$, and 
$\xi_j$ is the eigenvalue associated with $p_j$ for $j=1, \dots, n$.
Since $\|a\|=1$, there exists $m \in \{1, \ldots, n\}$ such that $|\xi_{m}|=1$.
 Assume that  $\{E_{kj}\}$ matrix units associated to the orthonormal basis given by the $\{p_j\colon 1 \leq j \leq n\}$ 
 (i.e., $E_{jj}=p_{j}$ and $E_{kj}=p_{k} E_{kj} p_{j}$). 
Using (\ref{Eq1.20200324}) and orthogonality of projections at the second step
 and using the hypothesis in the statement of the lemma in the last step, we get
\begin{equation}\label{Eq2.20200324}
|\xi_{m} - \xi_{k}|= \| (\xi_{m} - \xi_{k}) E_{m k} \|
=
\| a E_{m k} - E_{m k} a \|<\ep.
\end{equation} 
Therefore, using (\ref{Eq1.20200324}) and $\sum_{j=1}^{n} p_j=1$ at the first step 
and using (\ref{Eq2.20200324}) at the last step,
\begin{equation*}
\| a - \xi_{m} 1\|
= \left\|\sum_{j=1}^{n} \xi_j p_j - \xi_{m} \sum_{j=1}^{n} p_j \right\|
= 
\left\|\sum_{j=1}^{n} (\xi_j - \xi_{m})p_j \right\|
=
\max_{j} | \xi_j - \xi_{m} |<\ep.
\end{equation*}
This completes the proof.
\end{proof}
\begin{prp}\label{L_af_proper}
Assume the notation and choices in Construction~\ref{Ctn}.
\begin{enumerate}
\item\label{L_af_proper.a}
 The C*-algebra $A$ is stably finite, separable, simple, and has stable rank one.
\item\label{L_af_proper.b}
The diagram~(\ref{Eq_9411_LimDiag}) commutes. Moreover,
there is a unique action $\alpha\colon G \to \Aut (A)$ such that $\af = \dirlim \af^{(n)}$. 
\item\label{L_af_proper.c}
$\alpha$ is strictly approximately  inner. 
\item\label{L_af_proper.c'}
$\alpha$ is  pointwise outer. 
\item\label{L_af_proper.d}
$C^*(G, A, \alpha)$ is simple.
\setcounter{TmpEnumi}{\value{enumi}}
\end{enumerate}
\end{prp}
\begin{proof}
We prove (\ref{L_af_proper.a}).
Stable finiteness and separability of $A$ are immediate.
For simplicity, it is easy to check that
the hypotheses of Proposition 2.1(ii) of \cite{DNN92} hold.
Since the direct system in
Construction~\ref{Ctn}(\ref{Ctn_10})
has diagonal maps in the sense of Definition~2.1 of~\cite{ElHoTm}, it follows from Theorem~4.1 of~\cite{ElHoTm}
that $A$ has stable rank one.

We only prove the first part of (\ref{L_af_proper.b}). 
It is easy to check that 
the diagram~(\ref{Eq_9411_LimDiag}) commutes if and only if 
\begin{align*}
&
\varphi_{r(n)} \left(\theta \left(\left(z_g  \otimes 1_{\nu}\right )w\right) \otimes 1_{r(n)} \right)  
\psi_{r(n)}\left(1_{\nu} \otimes f (x_n)\right)
 \varphi_{r(n)} \left(\theta \left(w^*\left(z^*_g  \otimes 1_{\nu}\right )\right) \otimes 1_{r(n)} \right)  
 \\\notag 
 &  \hspace*{0em} {\mbox{}}
  =
\varphi_{r(n)} \left(\theta (w) \otimes 1_{r(n)} \right) 
\\\notag 
 &  \hspace*{3em} {\mbox{}} \cdot 
 \psi_{r(n)} \left(1_{\nu} \otimes \sigma_{r(n)} \left(z_g \otimes 1_{r(n)}\right) f(x_n) \sigma_{r(n)} \left(z^*_g \otimes 1_{r(n)}\right) \right)
\varphi_{r(n)} \left(\theta (w^*) \otimes 1_{r(n)} \right) 
\end{align*}
for all $n\in \Nz$, $x_n \in X_n$, $f\in C (X_{n}, M_{\nu r(n)})$, and  $g \in G$. 
By Lemma~\ref{Lem_Identification}, this formula is equivalent to the following:
\begin{align*}
&
\varphi_{r(n)} \left(\theta \left(\left(z_g  \otimes 1_{\nu}\right )w\right) \otimes 1_{r(n)} \right)  
\psi_{r(n)}\left(1_{\nu} \otimes f (x_n)\right)
 \varphi_{r(n)} \left(\theta \left(w^* \left(z^*_g  \otimes 1_{\nu}\right )\right) \otimes 1_{r(n)} \right) 
 \\\notag 
 &  \hspace*{.1em} {\mbox{}}
  =
\varphi_{r(n)} 
\left(
\theta \left(w \left(1_\nu \otimes z_g\right)\right) \otimes 1_{r(n)} 
\right)
 \psi_{r(n)} \left(  1_{\nu} \otimes f(x_n)  \right)
\varphi_{r(n)} 
\left(
\theta \left( \left(1_\nu \otimes z^*_g\right)w^*\right) \otimes 1_{r(n)}
\right). 
\end{align*}
This formula holds if and only if 
\[
\varphi_{r(n)} \left( \theta \left( (1_\nu \otimes z^*_g)w^*(z_g \otimes 1_\nu)w \right) \otimes 1_{r(n)}\right)
\in \left\{ 
\psi_{r(n)} \big(1_{\nu} \otimes M_{\nu r(n)}\big) \right\}^{'}
\]
for all $n \in \Nz$ and $g \in G$. 
This is true by Construction~\ref{Ctn}(\ref{Ctn_5}) and the fact that 
\[
\psi_{r(n)} \left((1_{\nu} \otimes M_{\nu r(n)})^{'}\right)=\psi_{r(n)}\left( M_{\nu} \otimes 1_{\nu r(n)}\right).
\]
Part (\ref{L_af_proper.c}) is immediate from Lemma~\ref{Lem_inlim_in_ap}. 

To prove (\ref{L_af_proper.c'}), let $g \in G \setminus \{1\}$ and 
suppose $\alpha_g$ is inner.
 So, there is a unitary $v \in A$ such that such that $\alpha_g (a) = v a v^*$ for all $a \in A$.
 Choose  $\ep \in (0, 1)$, $n \in \Nz$, and $u \in \U(A_n)$ such that 
 $\|u - v \|< \frac{\ep}{8}$.
  Using this, we get
\begin{equation}\label{Eq520200306}
 \|\Ad (u) - \alpha_g \| =\|\Ad (u) - \Ad (v) \| < \frac{\ep}{4}.
\end{equation}
Using (\ref{Ctn_12}) at the second step, 
we get,
for all $f \in A_n$ with $\| f \| = 1$ and all $x \in X_n$,
\begin{align*}
&\Big\|
\left[\Ad (u)(f)\right](x) - \left[\alpha_g (f)\right](x) 
\Big\|
\\
 & \hspace*{5em} {\mbox{}}=
 \Big\| u(x) f(x) u(x)^* 
 - 
 \sigma_{r(n)}\left(z_g \otimes 1_{r(n)}\right) f(x) \sigma_{r(n)}\left(z^*_g \otimes 1_{r(n)}\right)\Big\|
 < \frac{\ep}{4}.
\end{align*}
This relation implies
\[
\left\| \left[\sigma_{r(n)} \left(z^*_g \otimes 1_{r(n)}\right) u(x)\right] f(x) 
 - 
 f(x) \left[\sigma_{r(n)}\left(z^*_g \otimes 1_{r(n)}\right) u(x)\right]\right\|
 < \frac{\ep}{4}.
\]
for all $f \in C(X_n, M_{\nu r(n)})$ with $\| f \| = 1$ and all $x \in X_n$.
Applying Lemma~\ref{Approx_Commutnat} with $M_{\nu r(n)}$ in place of $M_n$,
$\sigma_{r(n)} \left(z^*_g \otimes 1_{r(n)}\right) u(x)$ in place of $b$,  and $\frac{\ep}{4}$ in place of $\ep$,
 we get $\lambda (x) \in \mathbb{T}$ such that 
\[
\left\| \sigma_{r(n)}\left(z^*_g \otimes 1_{r(n)}\right) u(x)  - \lambda (x) 1_{\nu r(n)}  \right\|< \frac{\ep}{4}.
\]
 This relation implies  
\begin{equation}\label{Eq7.20200304}
\left\|  u(x)  - \lambda(x) \sigma_{r(n)} \left(z_g \otimes 1_{r(n)}\right)\right\|< \frac{\ep}{4}.
\end{equation}
We denote $(e_{g, h})_{g, h \in G}$ by the standard system of matrix units in $M_{\nu} (\mathbb{C})$.
Set 
\begin{equation}
\label{Eq1.20200622}
b_0= \varphi_{r(n)} \left(\theta \left(w(e_{1, 1} \otimes 1_{\nu})w^*\right) \otimes 1_{r(n)}\right).
\end{equation} 
Define the continuous map $b \colon X_{n+1} \to M_{\nu r(n+1)}$ by
\begin{equation}\label{Eq12.20200306}
b(x)=
 \left( \begin{matrix}
    0 & & & &0 \\
    & &  \ddots      &      \\
    & & &  0 \\
 0  & & & &    b_0
 \end{matrix} \right)
 \in M_{\nu r(n+1)}.
\end{equation}
We further define the continuous map $h\colon X_n \to M_{\nu r(n)}$ by  
\begin{equation}\label{Eq12.20200306'}
h(x) = \sigma_{r(n)} \left(z_g \otimes 1_{r(n)}\right). 
\end{equation}

Now we claim, for all $x \in X_{n+1}$,
\begin{enumerate}
\setcounter{enumi}{\value{TmpEnumi}}
\item\label{CL.1}
$\Big\|
\left[\Ad \left(\Gamma_{n+1, n}(h)\right)(b)\right](x)
- 
\left[\alpha^{(n+1)}_{g} (b)\right](x) 
\Big\| =1.
$
\item\label{CL.2}
$\Big\| 
\left[\Ad \left(\Gamma_{n+1, n} (h)\right) (b)\right](x)
 - 
 \left[\Ad \left(\Gamma_{n+1, n}(u)\right) (b)\right] (x)
  \Big\| < \frac{\ep}{2}$.
\end{enumerate} 
We prove (\ref{CL.1}). 
By Construction~\ref{Ctn}(\ref{Ctn_10}) and (\ref{Eq12.20200306'}), we have
\begin{align}\label{EQ11.20200306}
\big[\Gamma_{n+1, n} (h) \big] (x)
  =
 \left( \begin{matrix}
    \sigma_{r(n)} \left(z_g \otimes 1_{r(n)}\right) & & & &0 \\
    & &  \ddots      &      \\
    & & & \sigma_{r(n)} \left(z_g \otimes 1_{r(n)}\right) \\
 0  & & & &   
c_{h} (x_n)
 \end{matrix} \right).
\end{align}
Using Construction~\ref{Ctn}(\ref{Ctn_10-}) and  (\ref{Eq12.20200306'}) at the first step
and using Lemma~\ref{Lem_Identification} at the second step,  we compute
\begin{align}
\label{Eq3.20200622}
&c_{h} (x_n)
\\\notag
&\hspace*{.4em} {\mbox{}}=
\varphi_{r(n)} \left(\theta(w) \otimes 1_{r(n)} \right)
\psi_{r(n)} \left(1_{\nu} \otimes \sigma_{r(n)}(z_g \otimes 1_{r(n)})\right)
\varphi_{r(n)} \left(\theta(w^*) \otimes 1_{r(n)} \right)
\\\notag
&\hspace*{.4em} {\mbox{}}=
\varphi_{r(n)} \left(\theta(w) \otimes 1_{r(n)} \right)
\varphi_{r(n)} \big(\theta \left(1_\nu \otimes z_g\right) \otimes 1_{r(n)}\big)
\varphi_{r(n)} \left(\theta(w^*) \otimes 1_{r(n)} \right)
\\\notag
&\hspace*{.4em} {\mbox{}}=
\varphi_{r(n)} \left(\theta \left(w \left(1_\nu \otimes z_g\right)w^*\right) \otimes 1_{r(n)} \right). 
\end{align}
Now, using (\ref{Eq1.20200622}) at the first step, using Construction~\ref{Ctn}(\ref{Ctn_5}) at the second step,
and
using the fact that  $z_g e_{h, h} z^*_g=e_{gh, gh}$ for all $g, h\in G$ at the last step, we get
\begin{align}\label{Eq5.20200702}
&
\varphi_{r(n)}\left(\theta \left(z_g \otimes 1_\nu\right) \otimes 1_{r(n)}\right)
 b_0 
\varphi_{r(n)}\left(\theta \left(z^*_g \otimes 1_\nu\right) \otimes 1_{r(n)}\right)
\\\notag
&\hspace*{8em} {\mbox{}}=
\varphi_{r(n)}\left(\theta \left((z_g \otimes 1_\nu) w (e_{1, 1} \otimes 1_{\nu})w^* (z^*_g \otimes 1_\nu)\right) \otimes 1_{r(n)}\right)
\\\notag
&\hspace*{8em} {\mbox{}}=
\varphi_{r(n)}\left(\theta \left(w (z_g \otimes z_g)(e_{1, 1} \otimes 1_{\nu})(z^*_g \otimes z^*_g)w^*\right) \otimes 1_{r(n)}\right)
\\\notag
&\hspace*{8em} {\mbox{}}=
\varphi_{r(n)}\left(\theta \left(w (z_g e_{1, 1} z^*_g \otimes 1_\nu)w^*\right) \otimes 1_{r(n)}\right)
\\\notag
&\hspace*{8em} {\mbox{}}=
\varphi_{r(n)}\left(\theta \left(w (e_{g, g} \otimes 1_\nu)w^*\right) \otimes 1_{r(n)}\right).
\end{align}
We use (\ref{Eq5.20200702}) and Construction~\ref{Ctn}(\ref{Ctn_12}) to compute
\begin{equation}\label{Eq2.20200306}
\left[\alpha^{(n+1)}_{g} (b)\right](x)
=
\left( \begin{matrix}
    0 & & & &0 \\
    & &  \ddots      &      \\
    & & &  0 \\
 0  & & & &   
 \varphi_{r(n)} \left( \theta \left(w ( e_{g, g} \otimes 1_{\nu})w^*\right) \otimes 1_{r(n)}\right)
 \end{matrix} \right).
\end{equation}
Now, using (\ref{Eq3.20200622}) at the first step and using (\ref{Eq1.20200622}) at the second step, we compute 
\begin{align}\label{Eq8.20200702}
&c_{h} (x_n) b_0 c_{h} (x_n)^*
\\\notag
& \hspace*{1em} {\mbox{}}=
\varphi_{r(n)}\left(\theta (w \left(1_{\nu} \otimes z_g\right) w^*)  \otimes 1_{r(n)}\right)
 b_0 \
\varphi_{r(n)} \left(\theta( w \left(1_{\nu} \otimes z^*_g\right) w^*)  \otimes 1_{r(n)} \right)
\\\notag
&\hspace*{1em} {\mbox{}}=
\varphi_{r(n)} \left(
\theta(w \left(e_{1, 1} \otimes 1_{\nu}\right) w^* ) \otimes 1_{r(n)}
\right).
\end{align}
Now, we use (\ref{Eq12.20200306}) and (\ref{EQ11.20200306}) to get
\begin{align}\label{Eq333.20200306}
\Big[ 
\Ad \big(\Gamma_{n+1, n} (h) \big) \left(b\right)
\Big](x)
 =
 \left( \begin{matrix}
    0 & & & &0 \\
    & &  \ddots      &      \\
    & & &  0 \\
 0  & & & &    
c_{h} (x_n) b_0 c_{h} (x_n)^*
 \end{matrix} \right).
\end{align}
Now, using (\ref{Eq2.20200306}), (\ref{Eq8.20200702}), and (\ref{Eq333.20200306}) at the first step, we get, for all $x \in X_{n+1}$,
\begin{align}\label{Eq7.20200308}
&\left\|
\left[\Ad \left(\Lambda_{n+1, n}(h)\right)(b)\right](x)
- 
\left[\alpha^{n+1}_{1} (b)\right](x) 
\right\|
\\\notag
&\hspace*{10em} {\mbox{}} 
=
\left\|
w \left(e_{1, 1} \otimes 1_{\nu}\right) w^*    
- 
w  \left( e_{g, g} \otimes 1_{\nu}\right) w^*
\right\|
\\\notag
&\hspace*{10em} {\mbox{}} =
\left\|
 e_{1, 1} \otimes 1_{\nu}   
- 
 e_{g, g} \otimes 1_{\nu}
\right\|
 =
\left\|
 e_{1, 1}    
- 
 e_{g, g}
\right\|=1.
\end{align}
This completes the proof of (\ref{CL.1}).

We prove (\ref{CL.2}).
Using $\| b_0 \|\leq 1$, $\lambda(x_n)\in \mathbb{T}$, and (\ref{Eq7.20200304}) at the second step,
we estimate,
\begin{align*}
&\left\|
c_{u} (x_n) b_0 c_{u} (x_n)^* - c_{h} (x_n) b_0 c_{h} (x_n)^*
\right\|
\\
&\hspace*{4em} {\mbox{}}
\leq
\left\| w \otimes 1_{r(n)}\right\| 
\cdot
\left\|1_{\nu} \otimes u(x_n) - 1_{\nu} \otimes \lambda(x_n) h(x_n) \right\|
\cdot
\left\|w^* \otimes 1_{r(n)} \right\|
\\\notag 
&\hspace*{13em} {\mbox{}} \cdot
\big\|
b_0
\big\|
\cdot
\big\|
\big[
\left(w \otimes 1_{r(n)}\right) \left(1_{\nu} \otimes u(x_n)\right)  \left(w^* \otimes 1_{r(n)}\right)
\big]^* 
\big\|
\\\notag
& \hspace*{6em} {\mbox{}}+
\left\|
\left(w \otimes 1_{r(n)}\right) \left(1_{\nu} \otimes \lambda(x_n) h(x_n)\right) \left(w^* \otimes 1_{r(n)}\right)
\right\|
\cdot
\big\|
b_0
\big\|
\cdot
\left\|w \otimes 1_{r(n)} \right\|
\\\notag 
&\hspace*{13em} {\mbox{}} \cdot
\big\|1_{\nu} \otimes u^*(x_n) - 1_{\nu} \otimes \overline{\lambda(x_n)} h^*(x_n) \big\|
\cdot
\big\|w^* \otimes 1_{r(n)} \big\|
\\\notag
&\hspace*{4em} {\mbox{}}<
\frac{\ep}{4} + \frac{\ep}{4}=\frac{\ep}{2}.
\end{align*}
Using this at the last step, we estimate, for all $x \in X_{n+1}$,
\begin{align*}
&\Big\| 
\left[\Ad \left(\Gamma_{n+1, n} \left(h\right)\right) (b)\right](x)
 - 
 \left[\Ad \left(\Gamma_{n+1, n}(u)\right) (b)\right] (x)
\Big\|
\\
&\hspace*{8em} {\mbox{}} =
\Big\|
c_{h} (x_n) b_0 c_{h} (x_n)^*
 -
c_{u} (x_n) b_0 c_{u} (x_n)^*
\Big\|
< 
\frac{\ep}{2}.
\end{align*}
This complete the proof of the claim.

Using (\ref{CL.1}) at the first step and using (\ref{CL.2}) and (\ref{Eq520200306}) at the third step,
 we get, for all $x \in X_{n+1}$,
\begin{align*}
&
1= 
\Big\|
\left[\Ad \left(\Gamma_{n+1, n}(h)\right)(b)\right](x)
- 
\left[\alpha^{(n+1)}_{g} (b)\right](x) 
\Big\|
\\
&\hspace*{.5em} {\mbox{}} \leq 
\Big\| 
\left[\Ad \left(\Gamma_{n+1, n} (h)\right) (b)\right](x)
 - 
 \left[\Ad \left(\Gamma_{n+1, n}(u)\right) (b)\right] (x)
  \Big\|
 \\&\hspace*{4em} {\mbox{}}+
\Big\|
\left[\Ad \left(\Gamma_{n+1, n}(u)\right) (b)\right] (x) 
- 
\left[\alpha^{(n+1)}_{g} (b)\right](x) 
\Big\|
\\\notag
&\hspace*{.5em} {\mbox{}}<
\frac{\ep}{2}+ \frac{\ep}{4} < \ep <1.
\end{align*}
This is a  contradiction.

Part (\ref{L_af_proper.d}) follows from simplicity of $A$, pointwise outeness of $\alpha$, and Theorem~3.1 of \cite{Ks1}.
\end{proof}
\begin{cor}\label{C_9422_IsomsInExample}
Assume the notation and choices in
Construction~\ref{Ctn}.
Then the maps
$\Cu (A) \to \Cu \left(C^*(G, A, \alpha)\right)$
and $\W (A) \to \W \left(C^*(G, A, \alpha)\right)$
are  isomorphisms onto their ranges.
\end{cor}

\begin{proof}
It is  essentially immediate from Theorem~\ref{WC_injectivity}
 and Proposition~\ref{L_af_proper}(\ref{L_af_proper.c}).
\end{proof}
Now we compute precisely the radius of comparison of the C*-algebra $A$ in the following theorem. 
\begin{thm}\label{rc.ctn.20200325}
Assume the notation and choices in Construction
\ref{Ctn} and~\Ntn{Eq3.Bott}.
Then $\rc (A) =\eta$.
\end{thm}
The proof of Theorem~\ref{rc.ctn.20200325} requires some preparation.
\begin{ntn}\label{Eq3.Bott}
Let $p \in M_2 \big(C (X_0)\big)$ be the Bott projection.
When $\nu>2$, by abuse of notion, we use $p$ to denote $p\oplus 0 \in M_{\nu} \big(C(X_0)\big)$.
Assuming the notation and choices in
Construction~\ref{Ctn}.
For $n \in \Nz$,
set $p_n =  \Gamma_{n, 0} (p) \in A_{n}$.
In particular, $p_0 = p$.
\end{ntn}
The set of normalized traces on a unital C*-algebra is denoted by $\T (A)$.
It was shown in \cite{Hag14} that $T(A)=\QT(A)$ for a exact unital C*-algebra $A$.
\begin{lem}\label{ProjectionRank22'}
Adopt the assumptions and notation
of Notation~\ref{Eq3.Bott}.
Let $n \in \Nz$ and for $j = 1, 2, \ldots, s (n)$
let $R_j^{(n)} \colon (S^2)^{s (n)} \to S^2$
be the $j$~coordinate projection.
Then:
\begin{enumerate}
\item
\label{ProjectionRank22_a'}
There are orthogonal projections $y_{n}, z_{n}$ in
$M_{\nu r (n)} \big( C (X_n) \big)$
such that:
\begin{itemize}
\item
$p_n = y_{n} + z_{n}$.
\item
 $y_{n}$ is
the direct sum of the projections
$p \circ R^{(n)}_{j}$ for $j = 1, 2, \ldots, s (n)$. 
\item
$z_{n}$ is a constant projection of rank $r (n) - s (n)$.
\end{itemize}
\item
\label{ProjectionRank22_b'}
For every $n \in \Nz$ and $\rho \in \T (A_{n})$, we have
$d_{\rho} (p_n) = \frac{1 }{\nu}$.
\end{enumerate}
\end{lem}
\begin{proof}
The formula holds for $n = 0$,
since $r (0) = s (0) = 1$.
Now assume that it is known for~$n$.
The important facts for the induction step are the following:
\begin{itemize}
\item
$\Gamma_{n + 1, n} ( y_{n})$
is the direct sum of the projections
$p \circ R^{(n)}_{j} \circ P^{(n)}_{k}$
for $j = 1, 2, \ldots, s (n)$ and $k = 1, 2, \ldots, d (n + 1)$
and a constant projection of rank
$\nu s (n)$.
\item
$\Gamma_{n + 1, n} ( z_{n})$
is a constant projection of rank
\[
[r (n) - s (n)] d(n+1) + \nu [r(n) - s(n)].
\]
\end{itemize}
Putting these together,
we get that $\Gamma_{n + 1, n} (p_n)$ is the direct sum of
$y_{n + 1}$ as described
and a constant function of rank
\begin{align*}
[r (n) - s (n)] d(n+1) + \nu [r(n) - s(n)] + \nu s (n)
&=
r (n) [ d(n+1) + \nu ] - s(n) d(n+1)
\\&=
r (n) l(n+1) + s(n) d(n+1) 
\\&
= r(n+1) - s(n+1).
\end{align*}
This completes the induction.

To prove (\ref{ProjectionRank22_b'}),
let $\rho$ be extreme in $\T (A_{n})$.
Then there is $x \in X_{n}$  such that
$\rho = \tr_{\nu r (n)} \otimes \ev_{x}$.
Therefore
\begin{align*}
d_{\rho} (p_n)
&=
 \rho (p_n)
=
 \frac{1}{\nu r (n)} \rank (p_{n} (x))
=
\frac{s (n) + [r (n) - s (n)]}{\nu r (n)} 
=
\frac{1}{\nu}. 
\end{align*}
This completes the proof of~(\ref{ProjectionRank22_b'}).
\end{proof}
\begin{dfn}
Let $X$ be a  compact Hausdorff space.
A projection $q$ in $M_{\infty} ( C (X))$ is said to be \emph{trivial}
if there exists $n \in \Nz$ such that $q$ is Murray-von Neumann
equivalent to $1_{M_{n} (C (X))}$.
When $n = 0$, this means $p = 0$.
\end{dfn}
\begin{lem}\label{C_7808_BigRank2}
Adopt the assumptions and notation of Notation~\ref{Eq3.Bott}.
Let $n \in \Nz$ and let $e$ be a trivial projection in
$M_{\infty} (A_n) \cong M_{\infty} (C (X_n) )$.
If there exists $x \in M_{\infty} (A_n)$
such that $\| x e x^* - p_{n} \| < \frac{1}{2}$ then
$\rank (e) \geq  r (n) + s (n)$.
\end{lem}
\begin{proof}
The proof is essentially the same as that of
Corollary~6.13 (or Corollary~6.20)  of \cite{AGP19}. 
\end{proof}
\begin{proof}[Proof of Theorem~\ref{rc.ctn.20200325}]
We  use Construction~\ref{Ctn}(\ref{Ctn_9}), Proposition~\ref{Prp6.2.Tom06}(\ref{Prp6.2.Tom06.b}) and Corollary~1.2 of \cite{EN13} to get 
\begin{equation}
\label{Eq1.20200702}
\rc (A_n)\leq \frac{\dim X_n}{2 \nu r(n)} = \frac{s(n)}{\nu r(n)}.
\end{equation}
Using the fact that $A$ is easily seen to be residually stably finite
(that is, all of its quotients are stably finite) and Proposition 2.13(ii) of~\cite{AA20} at the first step,
using (\ref{Eq1.20200702}) at the second step, and using Lemma~\ref{L_Kap_20200325} at the third step, we get 
\[
\rc (A) \leq \liminf_{n \to \infty} \rc (A_n) = \lim_{n \to \infty} \frac{s(n)}{\nu r(n)} =\eta.
\]

Now, it suffices to prove that
$\rc (A) \geq \eta$.
Suppose $\lambda < \eta$.
We show that $A$ does not have $\lambda$-comparison.
Choose $n \in \N$ such that $1 / {r (n)} < \eta - \lambda$.
Choose $M \in \Nz$ such that
\begin{equation}\label{Eq1.20200426}
\lambda + \frac{1}{\nu}
 < \frac{M}{ \nu r (n)}
 < \eta + \frac{1}{\nu}.
 \end{equation}
Let $e \in M_{\infty} (A_n)$ be a trivial \pj{}
of rank~$M$.
By slight abuse of notation,
we use $\Gamma_{m, n}$ to denote the amplified map
from $M_{\infty} (A_n)$ to $M_{\infty} (A_m)$ as well.
For $m > n$, the rank of $\Gamma_{m, n} (e)$
is $M \cdot \frac{r (m)}{r (n)}$.
We claim that, for $m > n$,
\begin{equation}\label{Eq2.20200426}
\rank \big(\Gamma_{m, n} (e) \big) < r (m) +  s(m).
\end{equation}
To prove the claim, suppose that
$\rank \big( \Gamma_{m, n} (e) \big) \geq  r (m) +  s(m)$.
Then, by (\ref{Eq1.20200426}),
\[
 r (m) + s(m) \leq M \cdot \frac{r (m)}{r (n)} < \nu \left(\eta + \frac{1}{\nu}\right) r (m).
\]
Thus, $\frac{s(m)}{r(m)} < \nu \eta$. This contradicts Lemma~\ref{L_Kap_20200325}. 
So the claim follows.

Now, for any extreme tracial state $\rho$ on $A_m$
(and thus for any trace on $A$),
we get, using \ref{Eq1.20200426} at the third step and  using \Lem{ProjectionRank22'}(\ref{ProjectionRank22_b'})
in the last step,
\begin{equation*}
d_{\rho} \big(\Gamma_{m, n} (e)\big)
 = \rho \big(\Gamma_{m, n} (e)\big)
 = \frac{1}{\nu r (m)} \cdot M \cdot \frac{r (m)}{r (n)}
> \lambda + \frac{1}{\nu}
= 
 \lambda + d_{\rho} (p_{m}).
\end{equation*}
On the other hand, if
$\Gamma_{\infty, 0} (p) \precsim_A \Gamma_{\infty, n} (e)$,
then there exist some $m > n$
and $x \in M_{\infty} (A_m)$
such that 
\[
\|x \Gamma_{m, n} (e)x^* - p_m\| < \frac{1}{2}.
\]
Using this and Lemma~\ref{C_7808_BigRank2}, we get
\[
\rank \big(\Gamma_{m,n} (e)\big) \geq r (m) + s (m).
\]
This contradicts (\ref{Eq2.20200426}), and we have proved that
$A$ does not have $\lambda$-comparison.
\end{proof}
Although we proved in Proposition~\ref{L_af_proper}(\ref{L_af_proper.c'}) that $\alpha$ is pointwise outer,
we show that $\alpha$ is far from having Rokhlin property in the following theorem.
 We further compute precisely the radius of comparison of $C^*(G, A, \alpha)$ .
\begin{thm}\label{Thm_rc_cross}
Assume the notation and choices in Construction
\ref{Ctn} and~\Ntn{Eq3.Bott}. Then:
\begin{enumerate}
\item\label{rc.Cros.1}
$\rc \left(C^*(G, A, \alpha)\right)= \eta$.
\item\label{rc.Cros.2}
$\alpha$ doesn't have the weak tracial Rokhlin property.
\end{enumerate}
\end{thm}
\begin{proof}
To prove (\ref{rc.Cros.1}), we use
 Theorem~{\ref{rc.ctn.20200325}} at the first step and
 use Theorem~\ref{Thm_rc_In}(\ref{Thm_rc_In.a}) 
  at the second step to get
\[
\eta = \rc (A) \leq  \rc \left(C^*(G, A, \alpha)\right).
\]
So it suffices to show that $\rc \left(C^*(G, A, \alpha)\right)\leq \eta$. 
Since $G$ is finite, there are $m \in \mathbb{Z}>0$
and positive integers $t(1)\leq t(2) \leq \cdots\leq t(m)$ 
such that 
$C^{*}(G) \cong \bigoplus_{j=1}^{m} M_{t(j)}$
 and 
$\sum_{ j=1}^{m} t(j)^{2} =\nu$. 
We may assume $t(1) = 1$.
Since $\alpha^{(n)}$ is inner for all $n \in \N$, it follows from Example~9.6.4 of \cite{GKPT18} that 
\begin{equation}\label{Eq1.20200330}
C^*\left(G, A_n, \alpha^{(n)}\right) \cong C^*(G) \otimes A_n  \cong \bigoplus_{j=1}^{m} M_{t(j)} ( A_n).
\end{equation}
Then, by Theorem~9.4.34 of \cite{GKPT18},
\begin{equation}\label{Eq1.20200427}
C^*(G, A, \alpha) 
= 
C^*\left(G, \dirlim A_n, \alpha\right)
\cong 
\dirlim C^*\left(G, A_n, \alpha^{(n)}\right).
\end{equation}
Using (\ref{Eq1.20200330}) at the first step, 
using Proposition~\ref{Prp6.2.Tom06}(\ref{Prp6.2.Tom06.a}) at the second  step, 
using $t(1)=1$ and Construction~\ref{Ctn}(\ref{Ctn_9}) at the third step,
using Proposition~\ref{Prp6.2.Tom06}(\ref{Prp6.2.Tom06.b}) at the fourth  step, and
using  Corollary~1.2 of \cite{EN13} at the fifth step,
 we get, for all $n \in \N$,
\begin{align}\label{Eq2.20200330}
\rc \left(C^*\left(G, A_n, \alpha^{(n)}\right)\right) 
&=
 \rc \left( \bigoplus_{j=1}^{m} M_{t(j)} \otimes A_n\right)
\\\notag
&= 
\max_{1\leq t(j)\leq m} \Big( \rc \left( M_{t(j)} \otimes A_n\right) \Big)
\\\notag
&=
 \rc \Big(C \left(X_{n}, M_{\nu r (n)}\right)\Big)
\\\notag
&= 
 \frac{1}{\nu r (n)}\rc \left(C \left(X_{n}\right)\right)
\\\notag
&\leq 
\frac{\dim X_n}{2 \nu r (n)}
 = 
 \frac{2 s(n)}{2 \nu r (n)}
 =
  \frac{s(n)}{\nu r (n)}.
\end{align}
Therefore, using (\ref{Eq1.20200427}) at the first step, using Proposition~2.11(ii) of \cite{AA20} at the second step, using (\ref{Eq2.20200330}) third step,
and using Lemma~\ref{L_Kap_20200325} at the last step,
\begin{align*}
\rc \big(C^*(G, A, \alpha)\big) 
&= 
\rc \left(\dirlim C^*(G, A_n, \alpha^{(n)})\right)
\\&
\leq
\liminf_{n \to \infty} \rc \left(C^*(G, A_n, \alpha^{(n)})\right)
\\&
=
\frac{1}{\nu} \lim_{ n \to \infty} \frac{s(n)}{r (n)} = \eta.
\end{align*}

Part (\ref{rc.Cros.2}) follows from 
Theorem~\ref{rc.ctn.20200325}, Proposition~\ref{L_af_proper}(\ref{L_af_proper.c}),
and Proposition~\ref{Pr.Tracial.In.Ro}.
\end{proof}
By Construction~\ref{Ctn}, Proposition~\ref{L_af_proper}, Theorem~\ref{rc.ctn.20200325} and Theorem~\ref{Thm_rc_cross}, 
we then get the following corollary.
\begin{cor}
For every finite group~$G$ and for every  $\eta \in \left(0, \frac{1}{\card (G)}\right)$, 
there exist a simple separable unital AH~algebra $A$ with  stable rank one 
and a strictly approximately  inner action $\alpha \colon G \to \Aut (A)$ such that:
\begin{enumerate}
\item\label{Cor.AH1}
$\af$ is pointwise outer, and $\alpha$ doesn't have the weak tracial Rokhlin property.
\item\label{Cor.AH2}
$\rc (A) =\rc \left(C^*(G, A, \alpha)\right)= \eta$.
\end{enumerate}
\end{cor}
\section{Open problems}\label{Sec_Open_Problems}
\begin{qst}\label{Q_1_eq}
Let $G$ be a finite group,
let $A$ be
an infinite-dimensional stably finite simple unital C*-algebra,
and let $\alpha \colon G \to \Aut (A)$ be a tracially approximately inner action.
Does it follow that
\[
\rc (A) = \rc \big( \CGAa \big)?
\]

We hope that if $\alpha \colon G \to \Aut (A)$ is a strictly approximately inner action 
of a finite group $G$ on a simple unital C*-algebra, then
$\rc \big( \CGAa \big) \leq \rc (A)$.
We certainly do not know that this is true, and proving it seems very
difficult.
\end{qst}
\begin{qst}\label{Q_2_TrivOnTA}
Does there exist a  strictly approximately inner action of a nontrivial nonabelian finite group
 on a unital \ca{} which is not approximately representable?
\end{qst}
\begin{qst}\label{Q_3_TrivOnTA}
Does there exist a  tracially strictly approximately inner action of a nontrivial nonabelian finite group
 on an infinite-dimensional simple unital \ca{}  which is not strictly approximately inner and not tracially approximately representable?
\end{qst}
\begin{qst}\label{Q_3_TrivOnTA}
Let $\alpha$ be an action of a nontrivial finite group on a C*-algebra 
which is strictly approximately inner and strongly approximately inner.
Does it follow that $\alpha$ is approximately representable?
\end{qst}
%

\end{document}